\pdfoutput=1 
\documentclass[a4paper]{amsart} 

\usepackage{amssymb}
\usepackage{amsthm}
\usepackage{amsmath}
\usepackage{mathrsfs}
\usepackage{mathtools}
\usepackage{bm}

\usepackage{graphicx}

\usepackage{tikz}
\usetikzlibrary{decorations.pathmorphing}

\newtheorem{theorem}{Theorem}[section]
\newtheorem{lemma}[theorem]{Lemma}
\newtheorem{proposition}[theorem]{Proposition}
\newtheorem{corollary}[theorem]{Corollary}

\theoremstyle{definition}
\newtheorem{definition}[theorem]{Definition}
\newtheorem{example}[theorem]{Example}

\theoremstyle{remark}
\newtheorem{remark}[theorem]{Remark}

\theoremstyle{definition}
\newtheorem{claim}{Claim}[section]

\def\Z{{\mathbb{Z}}} 
\def\R{{\mathbb{R}}}

\newcommand{\tridot}{\ooalign{$\triangle$\crcr\hss$\cdot$\hss}}

\begin{document}

\title{Faces of maximal chain polytopes}



\author{Shinsuke Odagiri}
\address{Shumei University, 1-1 Daigaku-cho, Yachiyo-shi, Chiba, 276-0003, Japan}
\email{p-odagiri@mailg.shumei-u.ac.jp}




\begin{abstract}
The maximal chain polytope $\mathscr{M}(P)$ is associated with a finite poset $P$.
For a set of maximal chains $\mathcal{C}$, it is shown that the convex hull of all the points corresponding to elements of $\mathcal{C}$ is not a face of $\mathscr{M}(P)$
if and only if $\mathcal{C}$ has an incomplete guided crown structure.
Using this result, several examples, including the dimension of $\mathscr{M}(\bm{m} \times \bm{n})$, are calculated.
\end{abstract}

\maketitle

\section{Introduction}

We first review basic notations of partially ordered sets.

A {\it partially ordered set} (or a {\it poset}) $(P, \le)$ is a set $P$ endowed with an ordering relation $\le$ (by standard abuse of notation, we identify $P$ with the poset).
In this paper, we assume every poset is {\it finite}, i.e. $p:=\# P < \infty$, and set $P=\{1, 2, \dots, p\}$.
Let $Q$ be a subset of $P$. Then $(Q, \le_Q)$ is a poset with $x \le_Q y$ if and only if $x\le y$.
We call $Q$ an {\it induced subposet} of $P$.
Unless indicated otherwise, we regard every subset of $P$ as an induced subposet.
A {\it chain} $C$ is a subset of $P$ such that every two elements are comparable, i.e. either $x<y, x>y$ or $x=y$ holds for every $x, y \in C$.
Moreover, if there exists no larger chain containing $C$, we call $C$ a {\it maximal chain}.
A subset $Q \subset P$ is called an {\it antichain} if every two distinct elements of $Q$ are incomparable.
The set of all the maximal chains (resp. all the maximal antichains) of $P$ is denoted by $\mathrm{MaxChain}(P)$ (resp. $\mathrm{AntiChain}(P)$).

We write $(x, y)$ for the {\it open interval} $\{z \in P \mid x < z < y \}$, and $[x, y]$ for the {\it closed interval} $\{z \in P \mid x \le z \le y \}$.
By introducing the imaginary minimum element $-\infty$ and the imaginary maximum element $\infty$ of $P$,
let us denote the {\it down-set} and the {\it up-set} of $x$ by $(-\infty, x]:=\{z \in P \mid z \le x\}$ and $[x, \infty):=\{z \in P \mid x \le z \}$, respectively.\footnote{The minimum and the maximum element of a poset (if exists) is usually written as $\hat{0}$ and $\hat{1}$.
We wrote as $-\infty$ and $\infty$ to emphasise that they are not elements of the poset.}
Other intervals, such as the half-open interval, are analogously defined.

The {\it covering relation}, denoted by $\lessdot$, is defined by $x \lessdot y$ if $x<y$ and $x \le z < y$ imply $x=z$.
We then say that $x $ is {\it covered by} $y$ or $y$ {\it covers} $x$.
If $x$ is covered by $y$ as elements of a subset $Q \subset P$, we write $x \lessdot_Q y$, 

Now, we will associate polytopes with posets.
To each subset $Q \subset P$, we define $e_Q=\sum_{x \in Q} e_x \in \R^p$, where $e_x$ is the canonical unit coordinate vector of $\R^p$.
Let $\mathcal{Q}=\{Q_1, \dots, Q_k\}$ be a set of subsets of $P$.
We write $\mathrm{conv}(\mathcal{Q})$ or $\mathrm{conv}(Q_1, \dots, Q_k)$ for the convex hull of $\{e_{Q_i} \in \R^p \mid 1 \le i \le k \}$.

Take a function $f \in \R^P$. Note that there is a canonical bijection $\R^P \ni f \mapsto \sum_{x \in P} f(x)e_x \in \R^p$.
Let $f(Q):=\sum_{x \in Q} f(x)$.
Then the definition of chain polytopes is given as follows (cf. \cite{S1}):

\begin{definition}
The {\it chain polytope} $\mathscr{C}(P)$ of the poset $P$ is the set of functions $f \in \R^P$ satisfying
\begin{itemize}
 \item $0 \le f(x),\ \forall x \in P$,
 \item $f(C) \le 1$, for every chain $C \subset P$.
\end{itemize}
\end{definition}

Then $\mathscr{C}(P)$ is a convex polytope of $\R^p$ under the canonical bijection. 
However, little is known about its face structure.
\begin{enumerate}
 \item The vertices of $\mathscr{C}(P)$ are $\{e_Q \mid Q \in \mathrm{AntiChain}(P) \}$. Thus, we have $\mathscr{C}(P)=\mathrm{conv}(\mathrm{AntiChain(P)})$ (cf. \cite{S1}).
 \item The number of edges of $\mathscr{C}(P)$ is equal to that of the order polytope $\mathscr{O}(P)$.
          Given two distinct antichains $Q_1, Q_2$ of $P$, $\mathrm{conv}(Q_1, Q_2)$ forms an edge of $\mathscr{C}(P)$
          if and only if $Q_1 \triangle Q_2 :=(Q_1 \setminus Q_2) \cap (Q_2 \setminus Q_1)$ is connected in $P$ (cf. \cite{HLSS, S1}). 
 \item Let $Q_1, Q_2, Q_3$ be distinct antichains of $P$.
          Then $\mathrm{conv} (Q_1, Q_2, Q_3)$ forms a 2-face of $\mathscr{C}(P)$ if and only if 
          $Q_1 \triangle Q_2, Q_2 \triangle Q_3$ and $Q_3 \triangle Q_1$ are connected in $P$ (cf. \cite{M}). 
 \item The facets of $\mathscr{C}(P)$ are the following (cf. \cite{S1}): 
 \begin{enumerate}
  \item $\{(t_{x_i}) \in \R^p \mid t_{x_k}=0\}$ for all $x_k \in P$,
  \item $\{(t_{x_i}) \in \R^p \mid \sum_{x \in C} t_x=1\}$, where $C \in \mathrm{MaxChain}(P)$.
 \end{enumerate}
\end{enumerate}

\begin{definition}
The {\it maximal chain polytope} $\mathscr{M}(P)$ of the poset $P$ is defined by $\mathscr{M}(P)=\mathrm{conv}(\mathrm{MaxChain}(P))$.
\end{definition}

In this paper, we shall determine the face structure of $\mathscr{M}(P)$ purely combinatorically;
Let $\mathcal{C}$ be a subset of $\mathrm{MaxChain}(P)$.
We will prove in Theorem \ref{mainthm} that $\mathrm{conv}(\mathcal{C})$ is not a face of $\mathscr{M}(P)$ if and only if $\mathcal{C}$ has an incomplete guided crown structure defined in Definition \ref{crownstr}.
It is also shown in Theorem \ref{mainthm2} that $\mathrm{conv}(\mathcal{C})$ is a $(\# \mathcal{C}-1)$-simplex and is a face of $\mathscr{M}(P)$ if and only if $\mathcal{C}$ does not have a guided crown structure.
Using these results and the basic properties of $\mathscr{M}(P)$, we present several results and examples in Chapter 3.
For instance, it is shown in Theorem \ref{thm321} that $\dim \mathscr{M}(\bm{m} \times \bm{n})=(m-1)(n-1)$.

Now, we will introduce another aspect of this work, which is applying posets to PERT/CPM in project scheduling (cf. \cite{CLM, S2}).

In project scheduling, projects are often modelled by activities and their properties such as time, cost, etc.
We will consider the so-called activity-on-node (AON) project network.
The set of activities $P$ has the ordering defined by $x < y$ if the activity $y$ cannot start before the activity $x$ is completed.
Suppose each activity $x$ takes time $f(x)$ to complete and consider the longest sequence of activities that must be completed to conclude a project.
Such a sequence is called the {\it critical path}.
That is to say, the critical path is the element $C \in \mathrm{MaxChain}(P)$ satisfying $f(C) \ge f(D)$ for every $D \in \mathrm{MaxChain}(P)$.
The time $f(C)$ to conclude a project is called the {\it earliest finishing time} (EFT).

Now, given an element $C$ of $\mathrm{MaxChain}(P)$,
let $\Gamma_C$ be the set of all points in $\R^p$ such that the corresponding function $f \in \R^P$ makes $C$ being the critical path.
Then $\R^p$ is divided into $\# \mathrm{MaxChain}(P)$ regions $\{\Gamma_C\}$.
Each $\Gamma_C$ can be viewed as a ``phase", and the phase transition may occur if the time cost $f(x)$ changes.
Let $C_1, \dots, C_k$ be elements of $\mathrm{MaxChain}(P)$.
Then the following question arises:
Do phases $\Gamma_{C_1}, \dots, \Gamma_{C_k}$ have a multicritical point?
More precisely, is there a point $f \in \R^P$ (here, we identify a point in $\R^p$ with a function in $\R^P$) such that $f(C_1)=\dots=f(C_k)>f(D)$ holds for all $D \not\in \mathrm{MaxChain}(P)$?
From Lemma \ref{lemma1}, such a point exists if and only if $\mathrm{conv}(C_1, \dots, C_k)$ is a face of $\mathscr{M}(P)$.
Therefore, we would like to know the face structure of $\mathscr{M}(P)$.

\begin{remark}
The earliest finishing time can be written as a tropical polynomial (where the addition is given by maximum) with all the coefficients being 0;
$$ F=\sum_{C \in \mathrm{MaxChain}(P)} \prod_{x \in C} t_x.$$
Then $\mathscr{M}(P)$ is the Newton polytope of $F$,
and the union of the boundaries of regions $\bigcup_{C \in \mathrm{MaxChain}(P)} \partial \Gamma_C$ is the tropical hypersurface $V(F)$
(cf. \cite{IMS, KO, MS}).
\end{remark}

\section{Guided crown structure and the existence of face}

The following lemma is well-known and plays an essential role throughout this paper.

\begin{lemma}\label{lemma1}
Let $\mathcal{C} \subset \mathrm{MaxChain}(P)$ be a nonempty set. Then the followings are equivalent:
\begin{enumerate}
 \item $\mathrm{conv}(\mathcal{C})$ is a face of $\mathscr{M}(P)$.
 \item There exists a function $f \in \R^P$ such that $f(C) \ge f(D)$ holds for every $C \in \mathcal{C}$ and $D \in \mathrm{MaxChain}(P)$,
         where the equality holds if and only if $D$ is an element of $\mathcal{C}$.
\end{enumerate}
\end{lemma}

\begin{proof}
(1)$\Rightarrow$(2): \
If $\mathcal{C}=\mathrm{MaxChain}(P)$, then the zero function $f(x)=0, \ \forall x \in P$ satisfies (2). 
Otherwise, there exists a supporting hyperplane
$$\left\{ (t_x)_{x \in P} \in \R^p \mid \sum_{x \in P} a_x t_x =b \right\}$$
of $\mathrm{conv}(\mathcal{C})$ such that $\mathscr{M}(P)$ is contained in the closed half-space
$$\left\{ (t_x)_{x \in P} \in \R^p \mid \sum_{x \in P} a_x t_x \le b\right\}.$$
Then the function $f(x)=a_x$ satisfies (2).

(2)$\Rightarrow$(1): \
Fix an element $C \in \mathcal{C}$. Then the hyperplane
$$\left\{ (t_x)_{x \in P} \in \R^p \mid \sum_{x \in P} f(x) t_x = f(C)\right\}$$
is the supporting hyperplane of $\mathrm{conv}(\mathcal{C})$.
\end{proof}

Let $Q$ be a subset of $P$.
For simplicity, the intersection of $Q$ and an interval of $P$ are represented using subscripts, e.g., $(x, y)_Q:=(x, y) \cap Q$.
Suppose $\mathcal{C} \subset \mathrm{MaxChain}(P)$.
For elements $x_1, x_2, \dots$ of $P$, we denote $\mathcal{C}_{x_1x_2\dots}:=\{C \in \mathcal{C} \mid x_i \in C,\ \forall i \}$.

\begin{definition}

\begin{enumerate}
 \item Let $\rho \ge 2$.
 $W=(\alpha_1, \beta_1, \alpha_2, \beta_2, \dots, \alpha_\rho, \beta_\rho) \ (\alpha_i, \beta_i \in P)$
 is called a {\it guided ($\rho$-)crown} of $\mathcal{C}$ if the following holds:
 \begin{itemize}
  \item $\alpha_i \lessdot \beta_i, \alpha_i \lessdot \beta_{i-1},\ \forall i \in \{1, \dots, \rho\}$, where $\beta_0=\beta_\rho$,
  \item $\alpha_i \neq \alpha_j, \beta_i \neq \beta_j$ for every distinct $i, j \in \{1, \dots, \rho\}$,
  \item $\mathcal{C}_{\alpha_i \beta_i} \neq \emptyset,\ \forall i \in \{1, \dots, \rho\}$.
 \end{itemize}
 For simplicity, we will identify the index $\rho$ with $0$, e.g., $\alpha_\rho=\alpha_0$.
 
 Let $W$ be a guided crown.
 $W$ is {\it complete} if for every $i \in \{1, \dots, \rho\}$,
 $$C_i \in \mathcal{C}_{\alpha_i\beta_i}, C_{i-1} \in \mathcal{C}_{\alpha_{i-1}\beta_{i-1}} \Rightarrow (-\infty, \alpha_i]_{C_i} \cup [\beta_{i-1}, \infty)_{C_{i-1}} \in \mathcal{C}.$$
 Otherwise, we say that $W$ is {\it incomplete}.

 \item $X=(\alpha_1, \beta_1, \alpha_2, \beta_2) \ (\alpha_i, \beta_i \in P)$ is called a {\it guided star} of $\mathcal{C}$ if the following holds:
 \begin{itemize}
  \item $\alpha_1 \not\lesseqgtr \alpha_2, \beta_1 \not\lesseqgtr \beta_2$,
  \item $\alpha_1 < \beta_1, \alpha_2 < \beta_2$,
  \item there exists $C_1 \in \mathcal{C}_{\alpha_1\beta_1}, C_2 \in \mathcal{C}_{\alpha_2\beta_2}$ such that $(\alpha_1, \beta_1)_{C_1} \cap (\alpha_2, \beta_2)_{C_2} \neq \emptyset$.
 \end{itemize}
 Let $X$ be a guided star.
 $X$ is {\it complete} if for every $C_1 \in \mathcal{C}_{\alpha_1\beta_1}$ and $C_2 \in \mathcal{C}_{\alpha_2\beta_2}$
 satisfying $(\alpha_1, \beta_1)_{C_1} \cap (\alpha_2, \beta_2)_{C_2} \neq \emptyset$,
 $$(-\infty, \gamma]_{C_1} \cup (\gamma, \infty)_{C_2} \in \mathcal{C} \ \text{and} \ (-\infty, \gamma]_{C_2} \cup (\gamma, \infty)_{C_1} \in \mathcal{C}$$
 holds for every $\gamma \in (\alpha_1, \beta_1)_{C_1} \cap (\alpha_2, \beta_2)_{C_2}$. Otherwise, we say that $X$ is {\it incomplete}.
\end{enumerate}
\end{definition}

\begin{remark}
When there is no possibility of confusion, guided $\rho$-crowns and guided stars will be referred to simply as $\rho$-crowns and stars, respectively.
\end{remark}

\begin{remark}
If $(\alpha_1, \beta_1, \alpha_2, \beta_2, \dots, \alpha_\rho, \beta_\rho)$ (resp. $(\alpha_1, \beta_1, \alpha_2, \beta_2)$) is a $\rho$-crown (resp. a star), then the same holds for its cyclic permutations
$$(\alpha_2, \beta_2, \dots, \alpha_\rho, \beta_\rho, \alpha_1, \beta_1), \dots, (\alpha_\rho, \beta_\rho, \alpha_1, \beta_1, \dots, \alpha_{\rho-1}, \beta_{\rho-1})$$
(resp. $(\alpha_2, \beta_2, \alpha_1, \beta_1)$).
\end{remark}

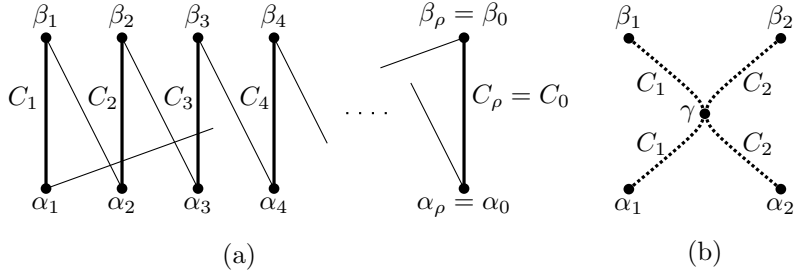
\begin{figure}[t]
\begin{minipage}{0.50\hsize}\centering
\begin{tikzpicture}
\coordinate (O) at (0,0) node at (O) [below] {$\alpha_1$} ;
\coordinate (A) at (0,2) node at (A) [above] {$\beta_1$} ;
\coordinate (B) at (1,0) node at (B) [below] {$\alpha_2$} ;
\coordinate (C) at (1,2) node at (C) [above] {$\beta_2$} ;
\coordinate (D) at (2,0) node at (D) [below] {$\alpha_3$} ;
\coordinate (E) at (2,2) node at (E) [above] {$\beta_3$} ;
\coordinate (F) at (3,0) node at (F) [below] {$\alpha_4$} ;
\coordinate (G) at (3,2) node at (G) [above] {$\beta_4$} ;
\coordinate (H) at (5.5,0) node at (H) [below] {$\alpha_\rho=\alpha_0$} ;
\coordinate (I) at (5.5,2) node at (I) [above] {$\beta_\rho=\beta_0$} ;
\coordinate (X) at (3.7,0.6) ;
\coordinate (Y) at (4.8,1.4) ;
\coordinate (S) at (3.95,1) ;
\coordinate (T) at (4.55,1) ;

\fill (O) circle (2pt) (A) circle (2pt) (B) circle (2pt) (C) circle (2pt) (D) circle (2pt) (E) circle (2pt) (F) circle (2pt) (G) circle (2pt) (H) circle (2pt) (I) circle (2pt);

\draw (O) -- (2.2,0.8) ;
\draw (4.4,1.6) -- (I) ;

\draw (A) -- (B) ;
\draw (C) -- (D) ;
\draw (E) -- (F) ;
\draw (G) -- (X) ;
\draw (Y) -- (H) ;
\draw[thick,loosely dotted] (S)--(T) ; 
\draw[very thick] (O) -- (A) ;
\draw[very thick] (B) -- (C) ;
\draw[very thick] (D) -- (E) ;
\draw[very thick] (F) -- (G) ;
\draw[very thick] (H) -- (I) ;

\node(C1) at (-0.3,1.2) {$C_1$};
\node(C2) at (0.75,1.2) {$C_2$};
\node(C3) at (1.75,1.2) {$C_3$};
\node(C4) at (2.75,1.2) {$C_4$};
\node(Cr) at (6.25,1.2) {$C_{\rho}=C_0$};

\end{tikzpicture}
\begin{center}(a) \end{center}
\end{minipage}
\begin{minipage}{0.45\hsize}\centering
\begin{tikzpicture}
\coordinate (O) at (0,0) node at (O) [below] {$\alpha_1$} ;
\coordinate (A) at (0,2) node at (A) [above] {$\beta_1$} ;
\coordinate (B) at (2,0) node at (B) [below] {$\alpha_2$} ;
\coordinate (C) at (2,2) node at (C) [above] {$\beta_2$} ;
\coordinate (D) at (1,1) node at (D) [left] {$\gamma$};

\fill (O) circle (2pt) (A) circle (2pt) (B) circle (2pt) (C) circle (2pt) (D) circle (2pt);

\draw[densely dotted, very thick] (0.99,1) .. controls (0.95,0.8) .. (O);
\draw[densely dotted, very thick] (0.99,1) .. controls (0.95,1.2) .. (A);
\draw[densely dotted, very thick] (1.01,1) .. controls (1.05,0.8) .. (B);
\draw[densely dotted, very thick] (1.01,1) .. controls (1.05,1.2) .. (C);

\node(C11) at (0.3,0.6) {$C_1$};
\node(C12) at (0.3,1.4) {$C_1$};
\node(C21) at (1.7,0.6) {$C_2$};
\node(C22) at (1.7,1.4) {$C_2$};

\end{tikzpicture}
\begin{center}(b) \end{center}
\end{minipage}
\caption{(a) shows a guided $\rho$-crown; (b) shows a guided star. Solid lines indicate covering relations. Thick lines indicate elements lying in the same maximal chain $C_i \in \mathcal{C}$.}
\label{fig1}
\end{figure}

\begin{definition}\label{crownstr}
$\mathcal{C}$ {\it has a guided $\rho$-crown} (resp. {\it a guided star}) if there exists a sequence of elements which is a $\rho$-crown (resp. a star).
If $\mathcal{C}$ has either a $\rho$-crown or a star, we say that $\mathcal{C}$ {\it has a guided crown structure}.
Suppose $\mathcal{C}$ has a guided crown structure.
$\mathcal{C}$ {\it has a complete guided crown structure} if every crown is complete, and every star is complete.
If not, $\mathcal{C}$ {\it has an incomplete guided crown structure}.
\end{definition}

\begin{remark}
When there is no possibility of confusion, we shall again omit the word ``guided''
(except for the statements in theorems).
For instance, guided crown structures will be referred to simply as crown structures. 
\end{remark}

Note that $\mathrm{MaxChain}(P)$ itself has either a complete crown structure or does not have a crown structure.

\begin{remark}\label{rem27}
If $\mathcal{C}$ has a crown structure, then $\#\mathcal{C} \ge 2$.
\end{remark}

\begin{remark}
$\mathcal{C}$ has a star if and only if there exists $C_1,\ C_2$ such that $C_1 \setminus C_2$ is disconnected as the induced subgraph of the Hasse diagram of $C_1$.
In such a case, $C_2 \setminus C_1$ is also disconnected as the induced subgraph of the Hasse diagram of $C_2$.
We call $\{C_1,\ C_2\}$ {\it star guides}.

Suppose $\mathcal{C}$ does not have a crown.
Then $\mathcal{C}$ does not have an incomplete crown structure if and only if for every star guides $\{C_1,\ C_2\}$ and every $\gamma  \in C_1 \cap C_2$,
$$(-\infty, \gamma]_{C_1} \cup (\gamma, \infty)_{C_2} \in \mathcal{C} \ \text{and} \ (-\infty, \gamma]_{C_2} \cup (\gamma, \infty)_{C_1} \in \mathcal{C}.$$
\end{remark}

For $C_1, C_2 \in \mathrm{MaxChain}(P)$, denote by $C_1 \tridot C_2$ the set $(C_1 \setminus C_2) \cup (C_2 \setminus C_1)$
which we regard as the induced subposet of the cover preserving subposet $C_1 \cup C_2$ of $P$.
That is to say, $x <_{C_1 \tridot C_2} y$ if and only if there exists a covering sequence 
$x=z_1 \lessdot_P \dots \lessdot_P z_k=y$, where each $z_i$ is an element of $C_1 \cup C_2$.
Especially, if $x, y \in (C_1 \setminus C_2) \cup (C_2 \setminus C_1)$ satisfy $x \lessdot_P y$, then we have $x \lessdot_{C_1 \tridot C_2} y$.
Note that $C_1 \setminus C_2, C_2 \setminus C_1$ are both maximal chains of this poset.

\begin{proposition}\label{prop29}
The followings are equivalent:
\begin{enumerate}
 \item $\mathcal{C}$ has either a $2$-crown or a star.
 \item There exists $C_1, C_2 \in \mathcal{C}$ such that
          $$\mathcal{C}':=\{C_1 \setminus C_2, C_2 \setminus C_1\} \subset \mathrm{MaxChain}(C_1 \tridot C_2)$$
          has a $2$-crown.
\end{enumerate}
\end{proposition}

\begin{proof}
$(1) \Rightarrow (2)$:\
If $\mathcal{C}$ has a $2$-crown $(\alpha_1, \beta_1, \alpha_2, \beta_2)$,
then for any $C_1 \in \mathcal{C}_{\alpha_1 \beta_1}$ and $C_2 \in \mathcal{C}_{\alpha_2 \beta_2}$, $\mathcal{C}'$ also has a $2$-crown $(\alpha_1, \beta_1, \alpha_2, \beta_2)$.

Suppose $\mathcal{C}$ has a star $(\alpha_1, \beta_1, \alpha_2, \beta_2)$.
Then there exists $C_1 \in \mathcal{C}_{\alpha_1 \beta_1}, C_2 \in \mathcal{C}_{\alpha_2 \beta_2}$ and $\gamma \in (\alpha_1, \beta_1)_{C_1} \cap (\alpha_2, \beta_2)_{C_2}$.
For each $i \in \{1, 2\}$, let $\alpha'_i$ and $\beta'_i$ be the largest and the smallest element of $(-\infty, \gamma)_{C_i} \setminus C_{3-i}$ and $(\gamma, \infty)_{C_i} \setminus C_{3-i}$, respectively (since $\alpha_i \in (-\infty, \gamma)_{C_i} \setminus C_{3-i}$ and $\beta_i \in (\gamma, \infty)_{C_i} \setminus C_{3-i}$ holds, such elements do exist).
Then $(\alpha'_1, \beta'_1, \alpha'_2, \beta'_2)$ is a $2$-crown of $\mathcal{C}'$.

$(2) \Rightarrow (1)$:\
Let $(\alpha_1, \beta_1, \alpha_2, \beta_2)$ be a $2$-crown of $\mathcal{C}'$.
Without loss of generality,  we may assume $\alpha_1, \beta_1 \in C_1 \setminus C_2$.
Then $\alpha_2, \beta_2 \in C_2 \setminus C_1$ must hold because $\alpha_1, \alpha_2$ are incomparable and $\beta_1, \beta_2$ are also incomparable.

If there exists $\gamma \in (\alpha_i, \beta_i)_{C_i}$ for some $i \in \{1, 2\}$, then $\gamma$ is an element of $C_1 \cap C_2$.
Hence $\alpha_{3-i} < \gamma < \beta_{3-i}$ must hold because of the incomparability again.
Thus, $(\alpha_1, \beta_1, \alpha_2, \beta_2)$ is a star of $\mathcal{C}$.
This argument also implies that $(\alpha_1, \beta_1)_{C_1}=\emptyset$ if and only if $(\alpha_2, \beta_2)_{C_2}=\emptyset$.
In this case, $(\alpha_1, \beta_1, \alpha_2, \beta_2)$ is a $2$-crown of $\mathcal{C}$.
\end{proof}

\begin{example}\label{ex210}
Let $P_1, P_2$ and $P_3$ be the posets whose Hasse diagrams are shown in Figure \ref{fig2}.
Here, we consider crowns/stars up to cyclic permutations.
\begin{enumerate}
 \item[($P_1$)] Let $\mathcal{C}_1=\{1356, 1357, 2456, 2457\}$.
                     Then $\mathcal{C}_1$ has complete stars $(\alpha_1, 6, \alpha_2, 7)$ and $(\alpha_1, 7, \alpha_2, 6)$, where $\alpha_1 \in \{1, 3\}, \alpha_2 \in \{2, 4\}$.
                     Also, there is an incomplete $2$-crown $(1, 3, 2, 4)$.
                     Thus, $\mathcal{C}$ has an incomplete crown structure.
 \item[($P_2$)] $\mathcal{C}_2=\{125, 1368, 468, 478\}$ has an incomplete crown structure because of the incomplete $3$-crown $(2, 5, 3, 6, 4, 7)$. 
                      On the other hand, $\mathcal{C}_2'=\{1278, 135,\allowbreak 1368, 478 \}$ do not have a crown structure.
 \item[($P_3$)] $\mathcal{C}_3=\{146, 257, 38\}$ has an incomplete crown structure with the incomplete $5$-crown (which is the only crown/star) being $(1, 4, 2, 5, 3, 8, 4, 6, 5, 7)$.
                      Let $\mathcal{C}_3'=\{146, 256, 257, 38\}$.
                      Since $\mathcal{C}_3'$ includes $\mathcal{C}_3$, $\mathcal{C}_3'$ also has the same $5$-crown.
                      Moreover, set $C_2=256 \in \mathcal{C}_{25}, C_4=146 \in \mathcal{C}_{46}$ and $C_5=257 \in \mathcal{C}_{57}$.
                      Then we have $(-\infty, 5]_{C_5} \cup [6, \infty)_{C_4}=C_2$.
\end{enumerate}
\end{example}

\begin{figure}[t]
\begin{minipage}{0.3\hsize}\centering
\begin{tikzpicture}
\coordinate (O) at (0,0) node at (O) [below] {1} ;
\coordinate (A) at (1,0) node at (A) [below] {2} ;
\coordinate (B) at (0,1) node at (B) [left] {3} ;
\coordinate (C) at (1,1) node at (C) [right] {4} ;
\coordinate (D) at (0.5,1.5) node at (D) [left] {5} ;
\coordinate (E) at (0,2) node at (E) [above] {6} ;
\coordinate (F) at (1,2) node at (F) [above] {7} ;

\fill (O) circle (2pt) (A) circle (2pt) (B) circle (2pt) (C) circle (2pt) (D) circle (2pt) (E) circle (2pt) (F) circle (2pt);

\draw (O) -- (B) ;
\draw (O) -- (C) ;
\draw (A) -- (B) ;
\draw (A) -- (C) ;
\draw (B) -- (D) ;
\draw (C) -- (D) ;
\draw (D) -- (E) ;
\draw (D) -- (F) ;
\end{tikzpicture}
\begin{center}(a) $P_1$\end{center}
\end{minipage}
\begin{minipage}{0.3\hsize}\centering
\begin{tikzpicture}
\coordinate (O) at (0,0) node at (O) [below] {2} ;
\coordinate (A) at (1,0) node at (A) [below] {3} ;
\coordinate (B) at (2,0) node at (B) [below] {4} ;
\coordinate (C) at (0,1) node at (C) [above] {5} ;
\coordinate (D) at (1,1) node at (D) [above] {6} ;
\coordinate (E) at (2,1) node at (E) [above] {7} ;
\coordinate (X) at (0.5,-0.5) node at (X) [below] {1};
\coordinate (Y) at (1.5,1.5) node at (Y) [above] {8};

\fill (O) circle (2pt) (A) circle (2pt) (B) circle (2pt) (C) circle (2pt) (D) circle (2pt)  (E) circle (2pt) (X) circle (2pt)  (Y) circle (2pt);

\draw (X) -- (O) ;
\draw (X) -- (A) ;
\draw (O) -- (C) ;
\draw (O) -- (E) ;
\draw (A) -- (D) ;
\draw (A) -- (C) ;
\draw (B) -- (E) ;
\draw (B) -- (D) ;
\draw (D) -- (Y) ;
\draw (E) -- (Y) ;
\end{tikzpicture}
\begin{center}(b) $P_2$\end{center}
\end{minipage}
\begin{minipage}{0.3\hsize}\centering
\begin{tikzpicture}
\coordinate (A) at (0,0) node at (A) [below] {1} ;
\coordinate (B) at (1,0) node at (B) [below] {2} ;
\coordinate (C) at (2,0) node at (C) [below] {3} ;
\coordinate (D) at (0,1) node at (D) [left] {4} ;
\coordinate (E) at (1,1) node at (E) [right] {5} ;
\coordinate (F) at (0,2) node at (F) [above] {6} ;
\coordinate (G) at (1,2) node at (G) [above] {7} ;
\coordinate (H) at (2,2) node at (H) [above] {8} ;

\fill (A) circle (2pt) (B) circle (2pt) (C) circle (2pt) (D) circle (2pt) (E) circle (2pt)  (F) circle (2pt) (G) circle (2pt) (H) circle (2pt);

\draw (A) -- (D) ;
\draw (A) -- (G) ;
\draw (B) -- (D) ;
\draw (B) -- (E) ;
\draw (C) -- (E) ;
\draw (C) -- (H) ;
\draw (D) -- (F) ;
\draw (D) -- (H) ;
\draw (E) -- (F) ;
\draw (E) -- (G) ;
\end{tikzpicture}
\begin{center}(c) $P_3$\end{center}
\end{minipage}
\caption{Three posets.}
\label{fig2}
\end{figure}
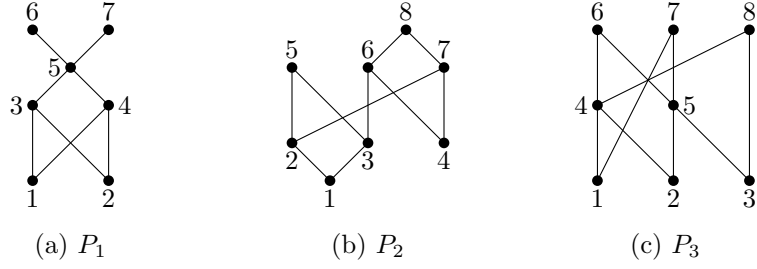

Suppose $W=(\alpha_1, \beta_2, \dots, \alpha_\rho, \beta_\rho)$ is a $\rho$-crown
and fix $C_i \in \mathcal{C}_{\alpha_i \beta_i}$ for each $i \in \{1, \dots, \rho\}$.
Then $C_i \neq C_j$ yields $i \neq j$.
However, as we saw with $P_3$ in the previous example, the converse does not hold in general.
Moreover, $(-\infty, \alpha_j]_{C_j} \cup [\beta_{j-1}, \infty)_{C_{j-1}}$ may equal to $C_i$ for some $i$.

\begin{proposition}\label{prop}
There exists $k \in \{1, \dots, \rho\}$ such that
$$(-\infty, \alpha_k]_{C_k} \cup [\beta_{k-1}, \infty)_{C_{k-1}} \neq C_i,\ \forall i \in \{1, \dots, \rho\}.$$
\end{proposition}

\begin{proof}
Suppose not.
Let $n_1=1$ and we define the infinite sequence $(n_i)_{i \ge 1}$ of elements of $\{1, \dots, \rho\}$ recursively as
$$C_{n_{i+1}}=(-\infty, \alpha_{n_i}]_{C_{n_i}} \cup [\beta_{n_i-1}, \infty)_{C_{n_i-1}}$$
(if $n_{i+1}$ is not unique, choose any one of them).
Since $\beta_{n_i-1} \neq \beta_{n_i}$, we have $C_{n_{i+1}} \neq C_{n_i}$.

Let $i \ge 2$ and suppose $C_{n_i} \neq C_{n_h}$ for all $h \in \{1, \dots, i-1\} $.
For $j \in \{1, \dots, i-1\}$, let $\alpha \in \bigcap_{j \le h \le i} C_{n_h}$ be the largest element satisfying
$$(-\infty, \alpha]_{C_{n_{h_1}}}=(-\infty, \alpha]_{C_{n_{h_2}}} \ \text{for all} \ h_1, h_2 \in \{j, \dots,  i\}$$
(such $\alpha$ exists because the minimum element of every $C_{n_h}$ is the same).
Then $\alpha=\alpha_{n_l}$ for some $l \in \{j, \dots, i-1\}$.
\begin{itemize}
 \item If $\alpha_{n_l} < \alpha_{n_i}$, then we have $\beta_{n_l} \in C_{n_j}$ and $\beta_{n_l-1} \in C_{n_i}$.
         Since $\beta_{n_l-1} \le \alpha_{n_i}$, we obtain $\beta_{n_l-1} \in C_{n_{i+1}}.$
 \item $\alpha_{n_l} = \alpha_{n_i}$ cannot happen because all the $\alpha_i$'s are different.
 \item If $\alpha_{n_l} > \alpha_{n_i}$, then we have $\beta_{n_i} \in C_{n_j}, \beta_{n_i-1} \in C_{n_{i+1}}$
\end{itemize}
Thus in either way, we have $C_{n_j} \neq C_{n_{i+1}}$.
Therefore, from the induction, all the $n_i$'s must be different, which is a contradiction because $\#\{1, \dots, \rho\}<\infty$.
\end{proof}

The following theorem is our main result.

\begin{theorem}\label{mainthm}
Let $\mathcal{C} \subset \mathrm{MaxChain}(P)$ be a nonempty subset.
Then $\mathrm{conv}(\mathcal{C})$ is not a face of $\mathscr{M}(P)$ if and only if $\mathcal{C}$ has an incomplete guided crown structure.
\end{theorem}

\begin{proof}
``If" part:\ 
Let $W=(\alpha_1, \beta_1, \dots, \alpha_{\rho}, \beta_{\rho})$ be an incomplete $\rho$-crown and fix $C_i \in \mathcal{C}_{\alpha_i\beta_i}$ for all $i$.
Since $W$ is incomplete, we may assume that $D_1:=(-\infty, \alpha_1]_{C_1} \cup [\beta_{\rho}, \infty)_{C_{\rho}}$ is not an element of $\mathcal{C}$.
Suppose that a function $f \in \R^P$ satisfy
$ f(C) \ge f(D)$ for every $C \in \mathcal{C}, D \in \mathrm{MaxChain}(P)$.
Then 
$$ \sum_{i=1}^\rho f(C_i)=\sum_{i=1}^\rho f((-\infty, \alpha_i]_{C_i} \cup [\beta_i, \infty)_{C_i})=\sum_{i=1}^\rho f((-\infty, \alpha_i]_{C_i} \cup [\beta_{i-1}, \infty)_{C_{i-1}}),$$
so every maximal chain $(-\infty, \alpha_i]_{C_i} \cup [\beta_{i-1}, \infty)_{C_{i-1}}$ is also taking the maximum value $f(C)$.
Thus, $D_1$ is a maximal chain attaining the maximum value and is not an element of $\mathcal{C}$.
Therefore, $\mathrm{conv}(\mathcal{C})$ is not a face of $\mathscr{M}(P)$ from Lemma \ref{lemma1}.
The case where $\mathcal{C}$ has an incomplete star can be shown similarly.

``Only if'' part:\
Let
$$\mathcal{F}=\{f \in \R^P \mid f(C) \ge f(D),\ \forall C \in \mathcal{C},\ \forall D \in \mathrm{MaxChain}(P)\}.$$
Note that $\mathcal{F}$ is nonempty since it contains a zero function.
For a function $f \in \mathcal{F}$, we define
$$\Omega(f)=\{D \in \mathrm{MaxChain}(P) \mid f(D) \ge f(D'),\ \forall D' \in \mathrm{MaxChain}(P)\}.$$
Since $\mathrm{conv}(\mathcal{C})$ is not a face, every element $f$ of $\mathcal{F}$ satisfies $\Omega(f) \supsetneq \mathcal{C}$ from Lemma \ref{lemma1}.
Among all the elements of $\mathcal{F}$, let $\tilde{f}$ satisfy $\#\Omega(\tilde{f}) \le \#\Omega(f)$ for every $f \in \mathcal{F}$,
and fix $\tilde{D} \in \Omega(\tilde{f}) \setminus \mathcal{C}$.

\begin{claim}
$$\tilde{D} \subset \bigcup_{C \in \mathcal{C}} C$$
\end{claim}

\begin{proof}[Proof of Claim A]
Suppose there exists $x_0 \in \tilde{D} \setminus \bigcup_{C \in \mathcal{C}} C$.
Define $\tilde{f}_0 \in \R^P$ as
$$\tilde{f}_0(x) = \left\{
\begin{array}{ll}
\tilde{f}(x_0)-1 & \text{if}\ x=x_0, \\
\tilde{f}(x) & \text{otherwise}.
\end{array}
\right.$$
Then $\tilde{f}_0$ is an element of $\mathcal{F}$ satisfying $\mathcal{C} \subset \Omega(\tilde{f}_0) \subset \Omega(\tilde{f}) \setminus \{\tilde{D}\}$
because $\tilde{f}_0(\tilde{D})=\tilde{f}(\tilde{D})-1$.
This contradicts the minimality of $\#\Omega(\tilde{f})$.
\end{proof}

Take a covering $\alpha \lessdot \beta$ of $\tilde{D}$ such that
\begin{itemize}
 \item there exists an element $C_1$ of $\mathcal{C}$  such that $(-\infty, \alpha]_{\tilde{D}} \subset C_1$,
 \item $(-\infty, \beta]_{\tilde{D}} \not\subset C$ for every $C \in \mathcal{C}$.
\end{itemize}

\begin{claim}
If there exists an element $C_2$ of $\mathcal{C}$ which contains both $\alpha$ and $\beta$, then $\mathcal{C}$ has an incomplete crown structure.
\end{claim} 

\begin{proof}[Proof of Claim B]
Since 
$$(-\infty, \alpha]_{C_2} \neq (-\infty, \alpha]_{\tilde{D}} = (-\infty, \alpha]_{C_1},$$
suppose $\alpha_2 \in C_2$ satisfy $[\alpha_2, \alpha]_{C_2} \not\subset C_1$ and $(\alpha_2, \alpha]_{C_2} \subset C_1$, and let $\gamma \in C_2$ be the covering of $\alpha_2$.
Then $\gamma$ is also an element of $C_1$.
Since $C_1$ is a maximal chain, there exists an element $\alpha_1$ of $C_1$ which is covered by $\gamma$, and an element $\beta_1$ of $C_1$ which covers $\alpha$.
Then $(\alpha_1, \beta_1, \alpha_2, \beta)$ is a star.
Moreover, $(-\infty, \gamma]_{C_1} \cup (\gamma, \infty)_{C_2} \not\in \mathcal{C}$ holds since no elements of $\mathcal{C}$ contain $(-\infty, \beta]_{\tilde{D}}$.
\end{proof}

Thus, we shall assume that no elements of $\mathcal{C}$ contain both $\alpha$ and $\beta$ after on.
Let $\Lambda=\{\alpha, \beta \} \times \{-1, 1\}$.
We assign the sequence $(\chi_x^i)_{i \ge 0}$ of 4-bit binary numbers to each element $x$ of $P$, where $\chi_x^i \in \{0, 1\}^\Lambda$, defined recursively as follows:
\begin{eqnarray*}
\chi_x^0(s, t ) &=& \left\{
\begin{array}{ll}
1 & \textrm{if}\ (x, s, t)=(\alpha, \alpha , -1)\ \textrm{or} \ (\beta, \beta, 1), \\
0 & \textrm{otherwise},
\end{array}
\right.\\
\chi_x^{i+1}(\alpha, -1) &=& \left\{
\begin{array}{ll}
1 & \textrm{if $\chi_x^i(\alpha, -1)=1$ or} \\
  &  \textrm{there exists a covering $x \lessdot y$ s.t. $\chi_y^i(\alpha, 1)=1$}, \\
0 & \textrm{otherwise},
\end{array}
\right.\\
\chi_x^{i+1}(\alpha, 1) &=& \left\{
\begin{array}{ll}
1 & \textrm{if $\chi_x^i(\alpha, 1)=1$ or there exists $C \in \mathcal{C}_x$} \\
  &  \textrm{and a covering $w \lessdot_C x$ s.t. $\chi_w^i(\alpha, -1)=1$}, \\
0 & \textrm{otherwise},
\end{array}
\right. \\
\chi_x^{i+1}(\beta, -1) &=& \left\{
\begin{array}{ll}
1 & \textrm{if $\chi_x^i(\beta, -1)=1$ or there exists $C \in \mathcal{C}_x$} \\
  &  \textrm{and a covering $x \lessdot_C y$ s.t. $\chi_y^i(\beta, 1)=1$}, \\
0 & \textrm{otherwise},
\end{array}
\right.\\
\chi_x^{i+1}(\beta, 1) &=& \left\{
\begin{array}{ll}
1 & \textrm{if $\chi_x^i(\beta, 1)=1$ or} \\
  &  \textrm{there exists a covering $w \lessdot x$ s.t. $\chi_w^i(\beta, -1)=1$}, \\
0 & \textrm{otherwise}.
\end{array}
\right.
\end{eqnarray*}

\begin{remark}
$\chi_x^i(s, t)=1$ means that $x$ is reachable within $i$ zigzag coverings starting from $s$ (the 0-th covering is $\alpha \lessdot \beta$),
with the covering of an element of $\mathcal{C}$ and any covering alternately repeating.
If $t=1$, then $x$ is the larger element of the zigzag, and $t=-1$ means that $x$ is the smaller element of the zigzag.
\end{remark}

\begin{claim}
There exists $x \in P$ satisfying
$$\chi_x^l(\alpha, -1)=\chi_x^l(\beta, -1)=1 \quad \textrm{or} \quad  \chi_x^l(\alpha, 1)=\chi_x^l(\beta, 1)=1$$
for some $l$.
\end{claim}

\begin{proof}[Proof of Claim C]
Suppose not.
If $\chi_x^i(s, t)=1$, then $\chi_x^{i+1}(s, t)=\chi_x^{i+2}(s, t)=\dots=1$.
Thus, there exists $m \ge 1$ such that $(\chi_x^i)$ is constant after on, i.e. $\chi_x^m(s, t)=\chi_x^{m+1}(s, t)=\dots$ for all $(x, s, t) \in P \times \Lambda$.

Let $C \in \mathcal{C}$.
If $x \in C$ satisfy $\chi_x^m(s, 1)=1$ for some $s \in \{\alpha, \beta\}$,
then either from the construction of $(\chi_x^i)$ or $\alpha \lessdot \beta$, $x$ covers an element of $P$.
Since $C$ is a maximal chain, there exists $w \in C$ which is covered by $x$.
Hence, we have
$$\chi_w^m(s, -1)=\chi_w^{m+1}(s, -1)=1\ (=\chi_x^m(s, 1)).$$
Similarly, if $\chi_x^m(s, -1)=1$, then there exists an element $y \in C$ which covers $x$, and we have $\chi_y^m(s, 1)=1$.
Thus, we have
\begin{eqnarray*}
\#\{x \in C  \mid \chi_x^m(\alpha, -1)=1\} &=& \#\{x \in C \mid \chi_x^m(\alpha, 1)=1\},\\
\#\{x \in C  \mid \chi_x^m(\beta, -1)=1\} &=& \#\{x \in C \mid \chi_x^m(\beta, 1)=1\}.
\end{eqnarray*}

Now consider $D \in \mathrm{MaxChain}(P)$.
Suppose $x \in D$ satisfy $\chi_x^m(\alpha, 1)=1$.
Then there exists $w \in D$ which is covered by $x$.
Hence, we have
$$\chi_w^m(\alpha, -1)=\chi_w^{m+1}(\alpha, -1)=1 \ (=\chi_x^m(\alpha, 1)).$$
Similarly, if $\chi_x^m(\beta, -1)=1$, then there exists an element $y \in D$ which covers $x$,
and we have $\chi_w^m(\beta, 1)=1$.
Thus, we have
\begin{eqnarray*}
\#\{x \in D  \mid \chi_x^m(\alpha, -1)=1\} &\ge& \#\{x \in D \mid \chi_x^m(\alpha, 1)=1\},\\
\#\{x \in D  \mid \chi_x^m(\beta, -1)=1\} &\le& \#\{x \in D \mid \chi_x^m(\beta, 1)=1\}.
\end{eqnarray*}
Moreover, these inequalities must be strict on $\tilde{D}$ because $\chi_\alpha^m(\beta, -1)=\chi_\beta^m(\alpha, 1)=0$ from the assumption.

Define $\tilde{f}' \in \R^P$ as 
$$\tilde{f}'(x)=\tilde{f}(x)-\chi_x^m(\alpha, -1)+\chi_x^m(\alpha, 1)+\chi_x^m(\beta, -1)-\chi_x^m(\beta, 1).$$
Then we have 
\begin{itemize}
 \item $\tilde{f}'(C)=\tilde{f}(C)$ for every $C \in \mathcal{C}$, 
 \item $\tilde{f}'(D) \le \tilde{f}(D)$ for every $D \in \mathrm{MaxChain}(P)$, 
 \item $\tilde{f}'(\tilde{D})<\tilde{f}(\tilde{D})$. 
\end{itemize}
Therefore, $\tilde{f}'$ is an element of $\mathcal{F}$ satisfying $\#\Omega(\tilde{f}')<\#\Omega(\tilde{f})$, which contradicts to the minimality of $\#\Omega(\tilde{f})$.
\end{proof}

Define $l$ as
\begin{multline*}
l=\min\{ i \in \Z_{> 0} \mid \chi_x^i(\alpha, -1)=\chi_x^i(\beta, -1)=1 \ \text{or} \\
\chi_x^i(\alpha, 1)=\chi_x^i(\beta, 1)=1, \ \exists x \in P\}.
\end{multline*}
Without loss of generality, let $\alpha' \in P$ satisfy $\chi_{\alpha'}^l(\alpha, -1)=\chi_{\alpha'}^l(\beta, -1)=1$.

Let
$$i_1=\min \{i \in \{1, \dots, l\} \mid \chi_{\alpha'}^i(\beta, -1)=1\}.$$
Then there exists $C_1 \in \mathcal{C}_{\alpha'}$ and $\beta_1 \in C_1$ satisfying $\alpha' \lessdot_{C_1} \beta_1$ and $\chi_{\beta_1}^{i_1-1}(\beta, 1)=1$.
Let
$$j_1=\min \{j \in \{1, \dots, i_1-1\} \mid \chi_{\beta_1}^j(\beta, 1)=1\}.$$
Then there exists $\alpha_2 \in P$ satisfying $\alpha_2 \lessdot \beta_1$ and $\chi_{\alpha_2}^{j_1-1}(\beta, -1)=1$.
Thus, we obtain a covering sequence $\alpha' \lessdot_{C_1} \beta_1 \gtrdot \alpha_2$.

By repeating this procedure starting from $\alpha_2$ and so on, we will eventually reach $\beta$.
We then obtain a covering sequence 
$$ \alpha' \lessdot_{C_1} \beta_1 \gtrdot \alpha_2 \lessdot_{C_2} \beta_2 \gtrdot \alpha_3 \lessdot_{C_3} \dots \gtrdot \alpha_k\lessdot_{C_k} \beta_k=\beta.$$
such that $\alpha', \alpha_2, \dots, \alpha_k$ are all different, and $\beta_1, \beta_2, \dots, \beta_{k-1}, \beta$ are also all different.

On the other hand, starting from $\chi_{\alpha'}^l(\alpha, -1)=1$, we can similarly obtain a covering sequence
$$ \alpha' \lessdot \beta'_1 \gtrdot_{C'_1} \alpha'_1 \lessdot \beta'_2 \gtrdot_{C'_2} \alpha'_2 \lessdot \beta'_3 \gtrdot_{C'_3} \dots \lessdot \beta'_{k'} \gtrdot_{C'_{k'}} \alpha'_{k'}=\alpha.$$
such that $\alpha', \alpha'_2, \dots, \alpha'_{k'-1}, \alpha$ are all different, and $\beta'_1, \beta'_2, \dots, \beta'_{k'}$ are also all different.
Moreover, $\alpha_i \neq \alpha'_j$ and $\beta_i \neq \beta'_j$ both holds for every $i, j$ because of the minimality of $l$.

Thus, $\mathcal{C}$ has a crown $(\alpha, \beta'_{k'}, \dots, \alpha'_1, \beta'_1, \alpha', \beta_1, \dots, \alpha_k, \beta)$ because
$$ \alpha  \lessdot_{C'_{k'}} \beta'_{k'} \gtrdot \dots \gtrdot \alpha'_1 \lessdot_{C'_1} \beta'_1 \gtrdot \alpha' \lessdot_{C_1} \beta_1 \gtrdot \dots \gtrdot \alpha_k \lessdot_{C_k} \beta \gtrdot \alpha. $$
Also, $(-\infty, \alpha]_{C'_{k'}} \cup [\beta, \infty)_{C_k}$ is a maximal chain not contained in $\mathcal{C}$ because it contains both $\alpha$ and $\beta$.
Therefore, $\mathcal{C}$ has an incomplete crown structure.
\end{proof}

%

\begin{theorem}\label{mainthm2}
 Let $\mathcal{C}=\{C_1, \dots, C_m\}$ be a nonempty subset of $\mathrm{MaxChain}(P)$.
\begin{enumerate}
 \item $\mathrm{conv}(\mathcal{C})$ is an $(m-1)$-simplex and is a face of $\mathscr{M}(P)$ if and only if $\mathcal{C}$ does not have a guided crown structure.
 \item $\mathrm{conv}(\mathcal{C})$ is not a simplex but is a face of $\mathscr{M}(P)$ if and only if $\mathcal{C}$ has a complete guided crown structure.
 \item $\mathrm{conv}(\mathcal{C})$ is not a face of $\mathscr{M}(P)$ if and only if $\mathcal{C}$ has an incomplete guided crown structure.
\end{enumerate}
\end{theorem}

\begin{proof}
(3) is already shown in Theorem \ref{mainthm}. We will show (2).

``If" part:\
We already know that $\mathrm{conv}(\mathcal{C})$ is a face of $\mathscr{M}(P)$.
Suppose that $\mathcal{C}$ contains a $\rho$-crown $W=(\alpha_1, \beta_1, \alpha_2, \beta_2, \dots, \alpha_\rho, \beta_\rho)$,
and fix $C_i \in \mathcal{C}_{\alpha_i\beta_i}$ for each $i$.
Let $D_i:=(-\infty, \alpha_i]_{C_i} \cup [\beta_i, \infty)_{C_{i-1}}$.
Then from the assumption, every $D_i$ is an element of $\mathcal{C}$.
Also from Proposition \ref{prop}, we have $D_k \not\in \{C_i\}_{1 \le i \le \rho}$ for some $k$.
Since $\sum_{1 \le i \le \rho} e_{C_i}=\sum_{1 \le i \le \rho} e_{D_i}$ holds, $\mathrm{conv}(\mathcal{C})$ is not a simplex.
The case that $\mathcal{C}$ has a star can be shown similarly.

``Only if" part:\
Since $\mathrm{conv}(\mathcal{C})$ is not a simplex, we have
$$ \dim \mathrm{conv}(C_1, \dots, C_k) =\dim \mathrm{conv}(C_1, \dots, C_{k+1}) $$
 for some $k$.
Then the affine hull of $\{e_{C_i}\}_{1 \le i \le k}$ must contain $e_{C_{k+1}}$.
Thus, every supporting hyperplane containing $\{e_{C_i}\}_{1 \le i \le k}$ also contains $e_{C_{k+1}}$.
From Theorem \ref{mainthm}, $\{C_1, \dots, C_k\}$ has an incomplete crown structure, so $\mathcal{C}$ also has a crown structure.
Since $\mathrm{conv}(\mathcal{C})$ is a face, the crown structure must be complete.
\end{proof}

\section{Face Structure of Maximal Chain Polytopes}

We assume that $\R^p$ is a Euclidean space throughout this section.

We first see that in simple cases, the face structure of $\mathscr{M}(P)$ can be easily obtained from the definition.

\begin{example}
\begin{enumerate}
 \item The maximal chain polytope of an $n$-chain $\bm{n}:=\{1, \dots, n\}$ is a point because $\# \mathrm{MaxChain}(\bm{n})=1$. 
 \item The maximal chain polytope of an $n$-antichain $A_n$ is an $(n-1)$-simplex because
          $\mathscr{M}(A_n)$ is the convex hull of all the $e_i$'s \ $(1 \le i \le n)$.
\end{enumerate}
\end{example}

The following definitions are well-known.

\begin{definition}
Let $P_1, P_2$ be posets.
\begin{enumerate}
 \item The {\it disjoint union} of $P_1$ and $P_2$, denoted by $P_1 + P_2$, is defined as the union $P_1 \cup P_2$ with the ordering relation being
 $$ x \leq_{P_1 + P_2} y \Longleftrightarrow \begin{cases} x, y \in P_1 \ \text{and}\ x \le_{P_1} y,\ \text{or} \\ x, y \in P_2 \ \text{and} \ x \le_{P_2} y. \end{cases}$$
 \item The {\it ordinal sum} of $P_1$ and $P_2$, denoted by $P_1 \oplus P_2$, is defined as the union $P_1 \cup P_2$ with the ordering relation being
 $$ x \leq_{P_1 \oplus P_2} y \Longleftrightarrow \begin{cases} x, y \in P_1 \ \text{and}\ x \le_{P_1} y,\ \text{or} \\ x, y \in P_2 \ \text{and} \ x \le_{P_2} y,\ \text{or} \\ x \in P_1 \ \text{and}\ y \in P_2. \end{cases}$$
 \item The {\it direct product} of $P_1$ and $P_2$, denoted by $P_1 \times P_2$, is defined as the set $P_1 \times P_2=\{(x, y) \mid x \in P_1, y \in P_2\}$ with the ordering relation being
 $$ (x_1, y_1) \leq_{P_1 \times P_2} (x_2, y_2) \Longleftrightarrow x_1 \leq_{P_1} x_2 \ \text{and} \ y_1 \leq_{P_2} y_2.$$
 Denote $P^n$ for the direct product $P \times P \times \dots \times P$ ($n$ times).
\end{enumerate}
\end{definition}

\begin{lemma}
Let $P_1,\ P_2$ be posets.
\begin{enumerate}
 \item $\mathscr{M}(P_1+P_2)=\mathscr{M}(P_1) \oplus \mathscr{M}(P_2):=$ the convex hull of \begin{flushright}$\left\{ \begin{pmatrix} t_1 \\ 0 \end{pmatrix} \in \R^{p_1+p_2} \mid t_1 \in \mathscr{M}(P_1)\right\} \cup \left\{ \begin{pmatrix} 0 \\ t_2 \end{pmatrix} \in \R^{p_1+p_2} \mid t_2 \in \mathscr{M}(P_2)\right\}$ \end{flushright}
 \item $\mathscr{M}(P_1 \oplus P_2)=\mathscr{M}(P_1) \times \mathscr{M}(P_2)$
 \begin{flushright} $:=\left\{ \begin{pmatrix} t_1 \\ t_2 \end{pmatrix} \in \R^{p_1+p_2} \mid t_1 \in \mathscr{M}(P_1), t_2 \in \mathscr{M}(P_2) \right\}$ \end{flushright}
\end{enumerate}
\end{lemma}

\begin{proof}
Since
\begin{align*}
\mathrm{MaxChain}(P_1+P_2) &= \mathrm{MaxChain}(P_1) \cup \mathrm{MaxChain}(P_2), \\
\mathrm{MaxChain}(P_1 \oplus P_2) &= \{ C_1 \cup C_2 \mid C_1 \in \mathrm{MaxChain}(P_1),\ C_2 \in \mathrm{MaxChain}(P_2)\}.
\end{align*}
holds, the lemma follows.
\end{proof}

\begin{example}
         Let $P_1$ be the poset we mentioned in Example \ref{ex210}.
         Then $P_1$ is isomorphic to $A_2 \oplus A_2 \oplus A_1 \oplus A_2$, and $\mathscr{M}(P_1)$ is a cube. See Figure \ref{fig3}(a).
\end{example}

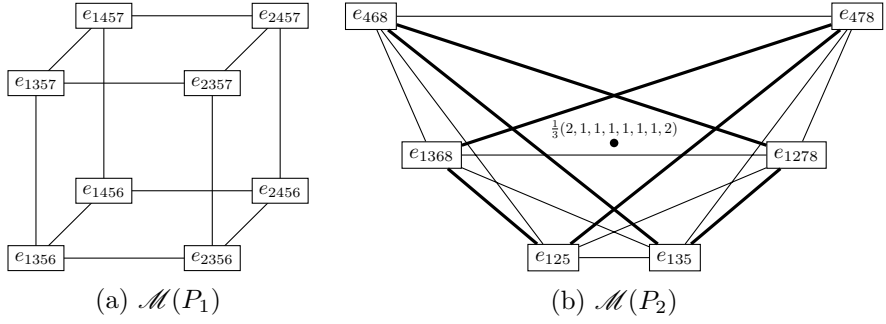
\begin{figure}[t]
\begin{minipage}{0.33\hsize}\centering
\begin{tikzpicture}[scale=0.58, transform shape] 
    \draw(-2,2,-2)--(-2,-2,-2)--(2,-2,-2);
    \draw(-2,-2,-2)--(-2,-2,2);
    \draw(-2,2,-2)--(-2,2,2)--(2,2,2)--(2,2,-2)--(-2,2,-2); 
    \draw(2,2,2)--(2,-2,2)--(2,-2,-2)--(2,2,-2); 
    \draw(-2,2,2)--(-2,-2,2)--(2,-2,2); 
    \node (A) [fill=white, draw,scale=1.3] at (-2,-2,2) {$e_{1356}$};
    \node (B) [fill=white, draw,scale=1.3] at (2,-2,2) {$e_{2356}$};
    \node (C) [fill=white, draw,scale=1.3] at (-2,-2,-2) {$e_{1456}$};
    \node (D) [fill=white, draw,scale=1.3] at (2,-2,-2) {$e_{2456}$};
    \node (E) [fill=white, draw,scale=1.3] at (-2,2,2) {$e_{1357}$};
    \node (F) [fill=white, draw,scale=1.3] at (2,2,2) {$e_{2357}$};
    \node (G) [fill=white, draw,scale=1.3] at (-2,2,-2) {$e_{1457}$};
    \node (H) [fill=white, draw,scale=1.3] at (2,2,-2) {$e_{2457}$};
\end{tikzpicture}
\begin{center}(a) $\mathscr{M}(P_1)$\end{center}
\end{minipage}
\begin{minipage}{0.60\hsize}\centering
\begin{tikzpicture}[scale=0.8, transform shape]
    \node (A1) [fill=white, draw] at (-1,0) {$e_{125}$};
    \node (B1) [fill=white, draw] at (-3,1.7) {$e_{1368}$};
    \node (C1) [fill=white, draw] at (4,4) {$e_{478}$};
    \node (A2) [fill=white, draw] at (1,0) {$e_{135}$};
    \node (B2) [fill=white, draw] at (3,1.7) {$e_{1278}$};
    \node (C2) [fill=white, draw] at (-4,4) {$e_{468}$};
    \coordinate (D) at (0, 1.9) node[scale=0.7] at (D) [above] {$\frac{1}{3}(2,1,1,1,1,1,1,2)$};
    \draw[very thick] (A1)--(B1) (B1)--(C1) (C1)--(A1);
    \draw[very thick] (A2)--(B2) (B2)--(C2) (C2)--(A2);
    \draw (A1)--(A2) (A1)--(B2) (A1)--(C2) (B1)--(A2) (B1)--(B2) (B1)--(C2) (C1)--(A2) (C1)--(B2) (C1)--(C2);
    \fill (D) circle (2pt);
\end{tikzpicture}
\begin{center} (b) $\mathscr{M}(P_2) $\end{center}
\end{minipage}
\caption{Maximal chain polytopes}
\label{fig3}
\end{figure}

\begin{proposition}\label{prop35}
Let $P$ be a poset and fix a nonempty subset $\mathcal{C} \subset \mathrm{MaxChain}(P)$ and a maximal chain $D \in \mathrm{MaxChain}(P) \setminus \mathcal{C}$.
If $\mathrm{conv}(\mathcal{C})$ is a face of $\mathscr{M}(P)$, then 
$\mathrm{conv}(\mathcal{C} \cup \{D\})$ is a cone with the base being $\mathrm{conv}(\mathcal{C})$.
\end{proposition}

\begin{proof}
Since $\mathrm{conv}(\mathcal{C})$ is a cone, there exists a supporting hyperplane containing $\mathrm{conv}(\mathcal{C})$ and not containing $e_{D}$.
\end{proof}

\begin{remark}
Assume that $\mathrm{conv}(\mathcal{C} \cup \{D\})$ is a face of $\mathscr{M}(P)$,
and let $\mathcal{C}_0$ be a proper nonempty subset of $\mathcal{C}$ such that $\mathrm{conv}(\mathcal{C}_0)$ is also a face of $\mathscr{M}(P)$.
Then the convex hull of $\mathrm{conv}(\mathcal{C}_0) \cup \{e_D\}$ is also a face of $\mathscr{M}(P)$ from Proposition \ref{prop35}.
Let us show this fact by constructing a function $\tilde{f} \in \R^P$ which takes its maximum value only at $\mathcal{C}_0 \cup \{D\}$ among all the elements of $\mathrm{MaxChain}(P)$.

Since $\mathrm{conv}(\mathcal{C}), \mathrm{conv}(\mathcal{C} \cup \{D\}), \mathrm{conv}(\mathcal{C}_0)$ are all faces of $\mathscr{M}(P)$,
there exists functions $f, g, f_0 \in \R^P$ satisfying
\begin{itemize}
 \item $f(C) \ge f(E)$ for every $C \in \mathcal{C}, E \in \mathrm{MaxChain}(P)$ with the equality holding if and only if $E \in \mathcal{C}$,
 \item $g(C) \ge g(E)$ for every $C \in \mathcal{C} \cup \{D\}, E \in \mathrm{MaxChain}(P)$ with the equality holding if and only if $E \in \mathcal{C} \cup \{D\}$,
 \item $f_0(C_0) \ge f_0(E)$ for every $C_0 \in \mathcal{C}_0, E \in \mathrm{MaxChain}(P)$ with the equality holding if and only if $E \in \mathcal{C}_0$.
\end{itemize}
For $B \in \mathcal{C}_0$, put $a=1/(f_0(B)-f_0(D)), b=1/(f(B)-f(D))$.
Then $a, b$ are both independent of the choice of $B$ and are both positive.
Fix a sufficiently small $\epsilon>0$ and define $\tilde{f}(x)=g(x)+\epsilon (af_0(x)-bf(x))$.
Then the following holds:
\begin{itemize}
 \item For $C_0, C_0' \in \mathcal{C}_0$, we have $f(C_0)=f(C_0'), g(C_0)=g(C_0')$ and $f_0(C_0)=f_0(C_0')$. Thus $\tilde{f}(C_0)=\tilde{f}(C_0')$.
 \item For $C_0 \in \mathcal{C}_0$, we have $g(C_0)=g(D)$ and $af_0(C_0)-bf(C_0)=af_0(D)-bf(D)$.
         Thus $\tilde{f}(C_0)=\tilde{f}(D)$.
 \item For $C_0 \in \mathcal{C}_0, C \in \mathcal{C} \setminus \mathcal{C}_0$, we have $f(C_0)=f(C), g(C_0)=g(C)$ and $f_0(C_0)>f_0(C)$.
         Thus $\tilde{f}(C_0)>\tilde{f}(C)$.
 \item For $C_0 \in \mathcal{C}_0, E \not\in \mathcal{C}\cup \{D\}$, we have $g(C_0)>g(E)$. Thus $\tilde{f}(C_0)>\tilde{f}(E)$ because $\epsilon$ is taken to be sufficiently small.
\end{itemize}
Therefore, $\tilde{f} \in \R^P$ satisfies 
$$\tilde{f}(C_0) \ge \tilde{f}(E) \ \text{for every} \ C_0 \in \mathcal{C}_0 \cup \{D\}, E \in \mathrm{MaxChain}(P)$$
with the equality holding if and only if $E \in \mathcal{C}_0 \cup \{D\}$.
\end{remark}

\begin{example}\label{ex37}
Let $P_4$ be the poset $\{1, \dots, 9\}$ such that its Hasse diagram is given in Figure \ref{fig4}(a).
Note that $P_4$ is isomorphic to $\bm{3}^2$.
$P_4$ does not have a crown, and $(2,7,3,8), (2,8,3,7)$ are the only stars up to cyclic permutations.
Put $\mathcal{C}=\{12579,12589,13579,13589\}$.
Then $\mathcal{C}$ has a complete crown structure, so $\mathrm{conv}(\mathcal{C})$ is a face of $\mathscr{M}(P_4)$.
Moreover, a simple calculation shows that $\mathrm{conv}(\mathcal{C})$ is a square (cf. Lemma \ref{lem311}).

Define $\mathcal{C}':=\mathcal{C} \cup \{12479\}$. Then $\mathcal{C}'$ does not have an incomplete crown structure,
so $\mathrm{conv}(\mathcal{C}')$ is a square pyramid-shaped face from Theorem \ref{mainthm2} and Proposition \ref{prop35}.
Similarly, $\mathrm{MaxChain}(P_4)=\mathcal{C}' \cup \{13689\}$ yields that
$\mathscr{M}(P_4)$ is a cone with the base being a square pyramid, as shown in figure \ref{fig4}(b).
In particular, $\dim \mathscr{M}(P_4)=4$.
\end{example}

\begin{figure}[t]
\begin{minipage}{0.3\hsize}\centering
\begin{tikzpicture}[scale=0.8]
\coordinate (O) at (0,0) node at (O) [below] {1} ;
\coordinate (A) at (-1,1) node at (A) [left] {2} ;
\coordinate (B) at (1,1) node at (B) [right] {3} ;
\coordinate (C) at (-2,2) node at (C) [left] {4} ;
\coordinate (D) at (0,2) node at (D) [right] {5} ;
\coordinate (E) at (2,2) node at (E) [right] {6} ;
\coordinate (F) at (-1,3) node at (F) [left] {7} ;
\coordinate (G) at (1,3) node at (G) [right] {8} ;
\coordinate (H) at (0,4) node at (H) [above] {9} ;

\fill (O) circle (2pt) (A) circle (2pt) (B) circle (2pt) (C) circle (2pt) (D) circle (2pt) (E) circle (2pt) (F) circle (2pt) (G) circle (2pt) (H) circle (2pt);

\draw (O) -- (A) ;
\draw (O) -- (B) ;
\draw (A) -- (C) ;
\draw (A) -- (D) ;
\draw (B) -- (D) ;
\draw (B) -- (E) ;
\draw (C) -- (F) ;
\draw (D) -- (F) ;
\draw (D) -- (G) ;
\draw (E) -- (G) ;
\draw (F) -- (H) ;
\draw (G) -- (H) ;
\end{tikzpicture}
\begin{center}(a) $P_4$\end{center}
\end{minipage}
\begin{minipage}{0.6\hsize}\centering
\begin{tikzpicture}[scale=0.8, transform shape]

\coordinate (A) at (-1.5,1.5) ;
\coordinate (B) at (1.5,1.5) ;
\coordinate (C) at (1.5,-1.5) ;
\coordinate (D) at (-1.5,-1.5) ;
\coordinate (E) at (4,0) ;
\coordinate (F) at (0,4) ;

\draw (A) -- (B) ;
\draw (B) -- (C) ;
\draw (C) -- (D) ;
\draw (D) -- (A) ;
\draw (A) -- (E) ;
\draw (B) -- (E) ;
\draw (C) -- (E) ;
\draw (D) -- (E) ;
\draw (A) -- (F) ;
\draw (B) -- (F) ;
\draw (C) -- (F) ;
\draw (D) -- (F) ;
\draw (E) -- (F) ;

\node (ZA) [fill=white, draw] at (A) {$e_{12589}$};
\node (ZB) [fill=white, draw] at (B) {$e_{13589}$};
\node (ZC) [fill=white, draw] at (C) {$e_{13579}$};
\node (ZD) [fill=white, draw] at (D) {$e_{12579}$};
\node (ZE) [fill=white, draw] at (E) {$e_{12479}$};
\node (ZF) [fill=white, draw] at (F) {$e_{13689}$};

\end{tikzpicture}
\begin{center}(b) $\mathscr{M}(P_4)$\end{center}
\end{minipage}
\caption{$P_4$ and its maximal chain polytope}
\label{fig4}
\end{figure}
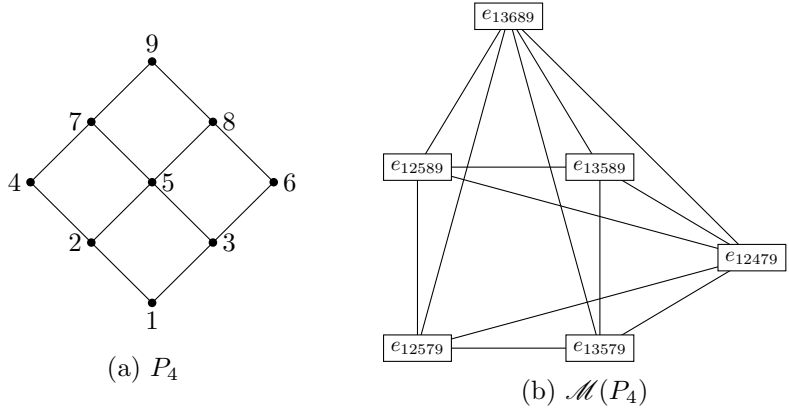

The following lemma holds because all the $\mathrm{conv}(\mathcal{C})$'s are 0/1-polytopes.

\begin{lemma}\label{lem38}
Let $S$ be a subset of $P$, and put $\mathcal{C} \subset \mathrm{MaxChain}(P)$.
Then $e_S \in \mathrm{conv}(\mathcal{C})$ if and only if $S \in \mathcal{C}$.
\end{lemma}

\begin{proof}
The ``If" part is obvious.

Assume $S \not\in \mathcal{C}$.
If $e_S \in \mathrm{conv}(\mathcal{C})$,
then
$$e_S=\sum_{C \in \mathcal{C}} \alpha_C e_C \quad \text{for some} \ \alpha_C \ge 0,\quad \sum_{C \in \mathcal{C}} \alpha_C=1$$
must hold.
Let $C' \in \mathcal{C}$ satisfy $\alpha_{C'} \neq 0$.
Since $S \neq C'$, either $C' \setminus S$ or $S \setminus C'$ is nonempty.
\begin{itemize}
 \item Suppose $C' \setminus S$ is nonempty, and let $x \in C' \setminus S$.
Then the $x$-th coordinate of $e_S$ must satisfy
$$0=\sum_{\substack{C \in \mathcal{C} \\ C \ni x}} \alpha_C \ge \alpha_{C'}>0.$$
 \item Suppose $S \setminus C'$ is nonempty, and let $y \in S \setminus C'$.
Then the $y$-th coordinate of $e_S$ must satisfy
$$1=\sum_{\substack{C \in \mathcal{C} \\ C \ni y}} \alpha_C \le 1-\alpha_{C'} < 1.$$
\end{itemize}
Therefore, a contradiction holds in both cases.
\end{proof}

\begin{corollary}
Every $e_C$ is a vertex of $\mathrm{conv}(\mathcal{C})$ for $C \in \mathcal{C}$.
\end{corollary}

\begin{proof}
Since $e_C \not\in \mathrm{conv}(\mathcal{C} \setminus \{C\})$ from the previous lemma, we obtain the corollary.
\end{proof}

A different proof of the following proposition can be found in \cite{KO}.

\begin{proposition}\label{prop310}
Let $C_1, C_2$ be two distinct maximal chains. The followings are equivalent:
\begin{enumerate}
 \item $\mathrm{conv}(C_1, C_2)$ is not an edge of $\mathscr{M}(P)$,
 \item $\{C_1, C_2\}$ has either a $2$-crown or a star,
 \item $\{C_1 \setminus C_2, C_2 \setminus C_1\} \subset \mathrm{MaxChain}(C_1 \tridot C_2)$ has a $2$-crown.
\end{enumerate}
\end{proposition}

\begin{proof}
$(1) \Rightarrow (2)$:\
Since $\mathrm{conv}(C_1, C_2)$ is not an edge of $\mathscr{M}(P)$, $\{C_1, C_2\}$ has an incomplete crown structure.
Suppose $(\alpha_1, \beta_1, \alpha_2, \beta_2, \dots, \alpha_\rho, \beta_\rho)$ is a $\rho$-crown ($\rho \ge 3$),
and we assume without loss of generality that $\alpha_1$ is an element of $C_1$.
Then $\alpha_1, \beta_1, \alpha_3,\ \beta_3, \dots \in C_1,\allowbreak \alpha_2, \beta_2, \alpha_4, \beta_4, \dots \in C_2$ must hold.
By applying cyclic permutation if necessary, let $\alpha_1$ be the smallest element of $\alpha_1, \alpha_3, \dots$.
Since all the $\alpha_i$'s  are distinct, $\alpha_1 \lessdot \beta_1 \le \alpha_3$ must hold.
Therefore, we have
$$ \alpha_2< \beta_1 \le \alpha_3 < \beta_2,$$
which is a contradiction because $\alpha_2 \lessdot \beta_2$ is a covering.

$(2) \Rightarrow (1)$:\
Neither $C_1$ nor $C_2$ contains both $\alpha_1$ and $\beta_2$, so the $2$-crown/star is incomplete.
Therefore, $\mathrm{conv}(C_1, C_2)$ is not an edge of $\mathscr{M}(P)$ from Theorem \ref{mainthm2}.

$(2) \Leftrightarrow (3)$ is already proven in Proposition \ref{prop29}.
\end{proof}

The following lemma is a property of 0/1-polytopes.

\begin{lemma}\label{lem311}
Let $S_1, S_2, S_3, S_4$ be four distinct subsets of $P$.
If $e_{S_1}, e_{S_2}, e_{S_3},e_{S_4}$ all lie on the same two-dimensional affine plane,
then $\mathrm{conv}(S_1, S_2, S_3, S_4)$ is a rectangle.
\end{lemma}

\begin{proof}
Put 
$$v^2=(v^2_x)=e_{S_2}-e_{S_1}, v^3=(v^3_x)=e_{S_3}-e_{S_1}, v^4=(v^4_x)=e_{S_4}-e_{S_1}.$$
Then none of the two vectors $v^2, v^3, v^4 \in \{-1, 0, 1\}^p$ are parallel.
Also for any $x \in P$, $\{v^2_x, v^3_x, v^4_x\}$ does not contain both $-1$ and $1$ because the difference between the $x$-coordinates of $e_{S_i}$ and $e_{S_j}$ is either $0$ or $1$ for any $i, j$.

From the assumption, there exists $a, b \in \R^\times$ such that $v^4=av^2+bv^3$.
However, $v^2_s \neq v^3_s$ for some $s$ yields $a, b \in \{-1, 1\}$.
If $a=b=-1$, we have $3e_{S_1}=e_{S_2}+e_{S_3}+e_{S_4}$, which occurs only when $S_1=S_2=S_3=S_4$.
A contradiction,
Thus we may assume $a=b=1$ by replacing $S_2, S_3$, and $S_4$ if necessary, 
Furthermore, either $v^2_x$ or $v^3_x$ must be 0 for every $x$, because $v^4_x \neq \pm 2$.
Therefore, we have $v^4=v^2+v^3$ and $v^2 \cdot v^3=0$, so $\mathrm{conv}(S_1, S_2, S_3, S_4)$ is a rectangle.
\end{proof}

\begin{remark}\label{rem312}
Suppose that $\{C_1,\ C_2\}$ has an incomplete crown structure.
Then Remark \ref{rem27} yields that $C_1 \neq C_2$.
Moreover, $\mathrm{conv}(C_1, C_2)$ is not an edge of $\mathcal{M}(P)$ from Theorem \ref{mainthm2}.
Thus, from Proposition \ref{prop310}, it has either a $2$-crown or a star.
\begin{itemize}
 \item If it has a $2$-crown $(\alpha_1, \beta_1, \alpha_2, \beta_2)\ (\alpha_i, \beta_i \in C_i)$, put
$$C_1'=(-\infty,\ \alpha_1]_{C_1} \cup [\beta_2, \infty)_{C_2}, C_2'=(-\infty,\ \alpha_2]_{C_2} \cup [\beta_1, \infty)_{C_1}.$$
 \item If it has a star $(\alpha_1, \beta_1, \alpha_2, \beta_2)\ (\alpha_i, \beta_i \in C_i)$, fix $\gamma \in (\alpha_1, \beta_1)_{C_1} \cap (\alpha_2, \beta_2)$ and put
$$C_1'=(-\infty,\ \gamma]_{C_1} \cup (\gamma, \infty)_{C_2}, C_2'=(-\infty,\ \gamma]_{C_2} \cup (\gamma, \infty)_{C_1}.$$
\end{itemize}
In both cases, $\mathrm{conv}(C_1, C_2, C_1', C_2')$ is a rectangle from Lemma \ref{lem311}.
\end{remark}

\begin{example}
Let $P_2$ be the poset that we mentioned in Example \ref{ex210}.
Then $P_2$ has two $3$-crowns $(2, 5, 3, 6, 4, 7), (2, 7, 4, 6, 3, 5)$ up to cyclic permutations.
Let $\mathcal{C} \in \mathrm{MaxChain}(P_2)$.Then $\mathrm{conv}(\mathcal{C})$ is not a face of $\mathscr{M}(P_2)$ if and only if
only one of the two sets $\{125, 1368, 478\}, \{1278,135,468\}$ is contained in $\mathcal{C}$.
Since $\frac{1}{3}(e_{125}+e_{1368}+e_{478})=\frac{1}{3}(e_{1278}+e_{135}+e_{468})$,
two triangles $\mathrm{conv}(125,1368,478)$ and $\mathrm{conv}(1278,135,468)$, not faces of $\mathscr{M}(P_2)$, intersect at a point $\frac{1}{3}(2, 1, 1, 1, 1, 1, 1, 2)$.
$\mathscr{M}(P_2)$ is the convex hull of such two triangles. See Figure \ref{fig3}(b).
\end{example}

\begin{lemma}\label{lem314}
Let $\mathcal{C}_1, \mathcal{C}_2 \subset \mathrm{MaxChain}(P)$.
Then $\mathcal{C}_1 \subset \mathcal{C}_2$ if and only if $\mathrm{conv}(\mathcal{C}_1) \subset \mathrm{conv}(\mathcal{C}_2)$.
\end{lemma}

\begin{proof}
The ``only if'' part is obvious, so we will show the ``if'' part.
Let $C \in \mathcal{C}_1$.
Then we have $e_C \in \mathrm{conv}(\mathcal{C}_1) \subset \mathrm{conv}(\mathcal{C}_2)$.
Therefore $C \in \mathcal{C}_2$ must hold from Lemma \ref{lem38}.
\end{proof}

Define
$$ \mathfrak{K}(P)=\{\mathcal{C} \in 2^{\mathrm{MaxChain}(P)} \setminus \{\emptyset\} \mid \mathcal{C} \ \text{does not have an incomplete crown str.}\} \cup \{\emptyset\}.$$
Then $\mathfrak{K}(P)$, endowed with the set inclusion, has a lattice structure.

From Lemma \ref{lem314}, we obtain the following theorem.

\begin{theorem}
There is an order-preserving bijection from $\mathfrak{K}(P)$ to the face lattice of $\mathscr{M}(P)$, sending $\mathcal{C}$ to $\mathrm{conv}(\mathcal{C})$.
\end{theorem}

\begin{corollary}\label{cor316}
Let $\mathcal{C}_1, \mathcal{C}_2 \in \mathfrak{K}(P)$ satisfy $\emptyset \neq \mathcal{C}_1 \subsetneq \mathcal{C}_2$.
Then $$\dim\mathrm{conv}(\mathcal{C}_1) < \dim\mathrm{conv}(\mathcal{C}_2).$$
Moreover, if $\mathcal{C}_1 \subset \mathcal{C}_2$ is a covering in $\mathfrak{K}(P)$, then
$$\dim\mathrm{conv}(\mathcal{C}_1)+1 = \dim\mathrm{conv}(\mathcal{C}_2).$$
\end{corollary}

For $\mathcal{C} \in 2^\mathrm{MaxChain}(P) \setminus \{\emptyset\}$,
let $\mathcal{C}^\star$ be the set of maximal chains obtained by adding maximal chains to $\mathcal{C}$ so that all the crowns/stars of $\mathcal{C}$ are complete, i.e.
\begin{multline*}
\mathcal{C}^\star := \mathcal{C} \cup \bigcup_{\substack{(\alpha_1, \beta_1, \dots, \alpha_{\rho}, \beta_{\rho}) \text{: a $\rho$-crown} \\ C_1 \in \mathcal{C}_{\alpha_1 \beta_1}, \dots, C_\rho \in \mathcal{C}_{\alpha_\rho \beta_\rho} \\ i \in \{1, \dots, \rho\}}} (-\infty, \alpha_i]_{C_i} \cup [\beta_{i-1}, \infty)_{C_{i-1}} \\
\cup \bigcup_{\substack{(\alpha_1, \beta_1, \alpha_2, \beta_2) \text{: a star} \\ C_1 \in \mathcal{C}_{\alpha_1 \beta_1}, C_2 \in \mathcal{C}_{\alpha_2 \beta_2} \\ \gamma \in (\alpha_1, \beta_1)_{C_1} \cap (\alpha_2, \beta_2)_{C_2} \\ i \in \{1, 2\}}} (-\infty, \gamma]_{C_i} \cup (\gamma, \infty)_{C_{3-i}}.$$
\end{multline*}
Then $\mathcal{C} \in \mathfrak{K}(P)$ if and only if $\mathcal{C}^\star=\mathcal{C}$.

Suppose $\mathcal{C} \not\in \mathfrak{K}(P)$.
Then every crown/star of $\mathcal{C}$, which is also a crown/star of $\mathcal{C}^\star$,  is complete in $\mathcal{C}^\star$.
However, $\mathcal{C}^\star$ may still have an incomplete crown structure, because $\mathcal{C}^\star$ may have a new crown/star.

\begin{example}
Let $P_5=\{1, 2, 3, 4, 5, 6\}$ be the poset such that its Hasse diagram is given in Figure \ref{fig5}(a).
Put $\mathcal{C}=\{14, 25, 26, 34\}$.
Then $\mathcal{C}$ has a $2$-crown $(1, 4, 2, 5)$, so we obtain $\mathcal{C}^\star=\{14, 15, 24, 25, 26, 34\}$.
However, $\mathcal{C}^\star$ still has an incomplete crown structure because of the new $3$-crown $(1, 5, 2, 6, 3, 4)$,
and we have $(\mathcal{C}^{\star})^\star=\mathrm{MaxChain}(P_5)$. See Figure \ref{fig5}.
\end{example}

\begin{figure}[t]
\begin{minipage}{0.3\hsize}\centering
\begin{tikzpicture}
\coordinate (O) at (0,0) node at (O) [below] {1} ;
\coordinate (A) at (1,0) node at (A) [below] {2} ;
\coordinate (B) at (2,0) node at (B) [below] {3} ;
\coordinate (C) at (0,1) node at (C) [above] {4} ;
\coordinate (D) at (1,1) node at (D) [above] {5} ;
\coordinate (E) at (2,1) node at (E) [above] {6} ;

\fill (O) circle (2pt) (A) circle (2pt) (B) circle (2pt) (C) circle (2pt) (D) circle (2pt) (E) circle (2pt);

\draw (O) -- (C) ;
\draw (O) -- (D) ;
\draw (A) -- (C) ;
\draw (A) -- (D) ;
\draw (A) -- (E) ;
\draw (B) -- (C) ;
\draw (B) -- (E) ;
\end{tikzpicture}
\begin{center}(a) Hasse diagram of $P_5$\end{center}
\end{minipage}
\begin{minipage}{0.3\hsize}\centering
\begin{tikzpicture}
\coordinate (O) at (0,0) node at (O) [below] {1} ;
\coordinate (A) at (1,0) node at (A) [below] {2} ;
\coordinate (B) at (2,0) node at (B) [below] {3} ;
\coordinate (C) at (0,1) node at (C) [above] {4} ;
\coordinate (D) at (1,1) node at (D) [above] {5} ;
\coordinate (E) at (2,1) node at (E) [above] {6} ;

\fill (O) circle (2pt) (A) circle (2pt) (B) circle (2pt) (C) circle (2pt) (D) circle (2pt) (E) circle (2pt);

\draw[very thick] (O) -- (C) ;
\draw (O) -- (D) ;
\draw (A) -- (C) ;
\draw[very thick] (A) -- (D) ;
\draw[very thick] (A) -- (E) ;
\draw[very thick] (B) -- (C) ;
\draw (B) -- (E) ;

\end{tikzpicture}
\begin{center}(b) $\mathcal{C}$\end{center}
\end{minipage}
\begin{minipage}{0.3\hsize}\centering
\begin{tikzpicture}
\coordinate (O) at (0,0) node at (O) [below] {1} ;
\coordinate (A) at (1,0) node at (A) [below] {2} ;
\coordinate (B) at (2,0) node at (B) [below] {3} ;
\coordinate (C) at (0,1) node at (C) [above] {4} ;
\coordinate (D) at (1,1) node at (D) [above] {5} ;
\coordinate (E) at (2,1) node at (E) [above] {6} ;

\fill (O) circle (2pt) (A) circle (2pt) (B) circle (2pt) (C) circle (2pt) (D) circle (2pt) (E) circle (2pt);

\draw[very thick] (O) -- (C) ;
\draw[very thick] (O) -- (D) ;
\draw[very thick] (A) -- (C) ;
\draw[very thick] (A) -- (D) ;
\draw[very thick] (A) -- (E) ;
\draw[very thick] (B) -- (C) ;
\draw (B) -- (E) ;
\end{tikzpicture}
\begin{center}(c) $\mathcal{C}^\star$\end{center}
\end{minipage}
\caption{$\mathcal{C}^\star$ has an incomplete crown structure}
\label{fig5}
\end{figure}
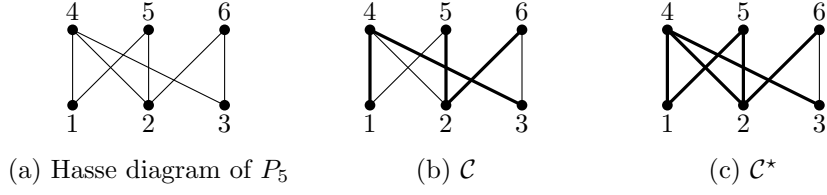

We have a sequence of inclusions $\mathcal{C} \subset \mathcal{C}^\star \subset \mathcal{C}^{2\star}:=(\mathcal{C}^\star)^\star \subset \mathcal{C}^{3\star}:=(\mathcal{C}^{2\star})^\star \subset \dots$.
Define $\overline{\mathcal{C}}:=\bigcup_{i} \mathcal{C}^{i\star}$.
Then $\overline{\mathcal{C}}=\mathcal{C}^{k\star}$ must hold for some sufficiently large $k$ because $\#\mathrm{MaxChain}(P)$ is finite.
From the definition of $\overline{\mathcal{C}}$, $\overline{\mathcal{C}}$ does not have an incomplete crown structure.
Therefore, $\mathrm{conv}(\overline{\mathcal{C}})$ is a face of $\mathscr{M}(P)$ from Theorem \ref{mainthm2}.

\begin{proposition}
$\mathrm{conv}(\overline{\mathcal{C}})$ is the minimum face of $\mathscr{M}(P)$ containing $\mathrm{conv}(\mathcal{C})$.
\end{proposition}

\begin{proof}
Suppose $\mathrm{conv}(\mathcal{D})$ is a face of $\mathscr{M}(P)$ satisfying $\mathrm{conv}(\mathcal{C}) \subset \mathrm{conv}(\mathcal{D})$.
Then from Theorem \ref{mainthm2}, we have $\mathcal{D} \in \mathfrak{K}$.
Moreover, we have $\mathcal{C} \subset \mathcal{D}$ from Lemma \ref{lem314}.
Therefore, $\overline{\mathcal{C}}=\mathcal{C}^{k\star} \subset \mathcal{D}^{k\star}=\mathcal{D}$ holds for some $k$,
so we obtain $\mathrm{conv}(\overline{\mathcal{C}}) \subset \mathrm{conv}(\mathcal{D})$.
\end{proof}

\begin{corollary}\label{cor319}
Suppose $\mathcal{C}, \mathcal{D} \in \mathfrak{K}(P)$ satisfy $\emptyset \neq \mathcal{C} \subsetneq \mathcal{D}$.
If $\overline{\mathcal{C} \cup \{D\}}=\mathcal{D}$ for any $D \in \mathcal{D} \setminus \mathcal{C}$, then
$\mathcal{D} \subset \mathcal{C}$ is a covering in $\mathfrak{K}(P)$.
\end{corollary}

\begin{proof}
Let $\mathcal{E} \in \mathfrak{K}(P)$ satisfy $\mathcal{C} \subsetneq \mathcal{E} \subset \mathcal{D}$.
Then for any $E \in \mathcal{E} \setminus \mathcal{C}$, we obtain $\mathcal{D}=\overline{\mathcal{C} \cup \{E\}} \subset \overline{\mathcal{E}} = \mathcal{E} \subset \mathcal{D}$.
\end{proof}

\begin{example}\label{ex320}
Consider the poset $P_3$ we mentioned in Example \ref{ex210}.
The followings are the only crowns/stars of $\mathrm{MaxChain}(P_3)$ up to cyclic permutations:
\begin{enumerate}
 \item crowns $(1, 4, 2, 5, 3, 8, 4, 6, 5, 7)$ and $(1, 7, 5, 6, 4, 8, 3, 5, 2, 4)$,
 \item stars $(1, 6, 2, 8)$ and $(1, 8, 2, 6)$,
 \item stars $(2, 6, 3, 7)$ and $(2, 7, 3, 6)$.
\end{enumerate}
Thus a subset $\mathcal{C} \subset \mathrm{MaxChain}(P_3)$ is not an element of $\mathfrak{K}(P_3)$ if and only if $\mathcal{C}$ satisfies one of the followings:
\begin{enumerate}
 \item $\mathcal{C}$ contains only one of the two sets
 $$\{14c, 25d, 38, a46, b57\}, \{24c, 35d, a48, b56, 17\}$$
 for some $a \in \{1, 2\}, b \in \{2, 3\}, c \in \{6, 8\}, d \in \{6, 7\}$,
 \item $\mathcal{C}$ contains only one of the two sets $\{146, 248\}, \{148, 246\}$, 
 \item $\mathcal{C}$ contains only one of the two sets $\{256, 357\}, \{257, 356\}$. 
\end{enumerate}
For instance, $\mathrm{conv}(146, 257, 38)$ is not a face of $\mathscr{M}(P_3)$ since $\{146, 257, 38\}$ satisfies (1) when $a=1, b=2, c=6$ and $d=7$.
Thus $\overline{\{146,257,38\}}$ must contain $\{146, 257, 38, 246, 357, 148, 256, 17\}$, and indeed they are equal.
Therefore, the minimum face containing $\mathrm{conv}(146, 257, 38)$ is $\mathrm{conv}(146, 257, 38, 246, 357, 148, 256, 17)$,
which is a convex hull of the triangle $\mathrm{conv}(146, 257, 38)$ and the $4$-simplex $\mathrm{conv}(246,\allowbreak  357,148, 256, 17)$ meeting at a point $\frac{1}{5}(2, 2, 1, 2, 2, 2, 2, 1)$.

Note that $\{146,148,17\} \in \mathfrak{K}(P_3)$.
Since
$$\overline{\{146,148,17,246\}}=\overline{\{146,148,17,248\}}=\{146,148,17,246,248\},$$
Corollary \ref{cor319} yields $\{146,148,17\} \subset \{146,148,17,246,248\}$ is a covering in $\mathfrak{K}(P_3)$.
Similarly thinking, we obtain the following maximal chain of elements of $\mathfrak{K}(P_3)$:
\begin{multline*}
\emptyset \subset \{146\} \subset \{146, 148\} \subset \{146, 148, 17\} \\
\subset \{146, 148, 17, 246, 248\} \subset \{146, 148, 17, 246, 248, 256\} \\
\subset \{146, 148, 17, 246, 248, 256, 257\} \subset \mathrm{MaxChain}(P_3).
\end{multline*}
Thus the dimension of $\mathscr{M}(P_3)$ is $6$ from Corollary \ref{cor316}.
\end{example}

\begin{theorem}\label{thm321}
 $$\dim \mathscr{M}(\bm{m} \times \bm{n}) =(m-1)(n-1)$$
\end{theorem}

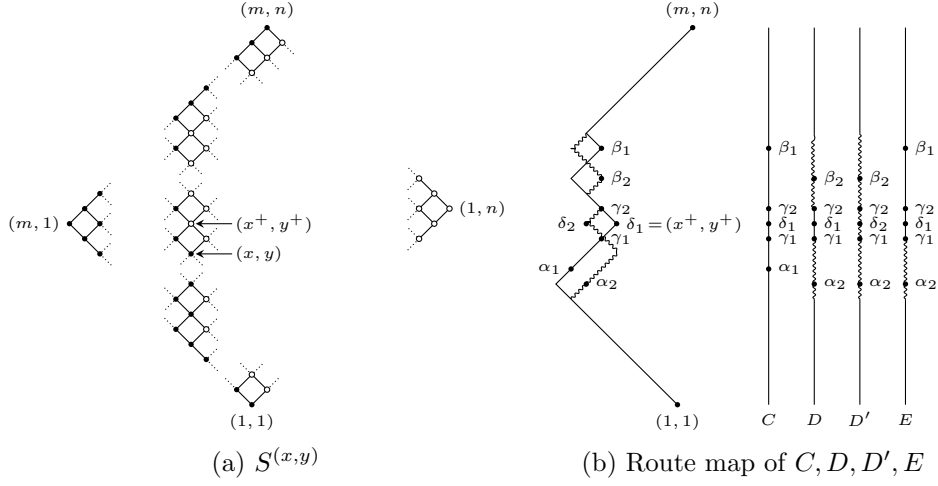
\begin{figure}[t]
\begin{tikzpicture}[scale=0.20]
\draw (0,0)--(-1,1)--(0,2)--(1,1)--cycle;
\draw (-3,3)--(-4,4) (-4,4)--(-5,5)--(-4,6)--(-5,7)--(-4,8)--(-3,7)--(-4,6)--(-3,5)--cycle;
\draw (-4,10)--(-5,11)--(-4,12)--(-5,13)--(-4,14)--(-3,13)--(-4,12)--(-3,11)--cycle;
\draw (-4,20)--(-5,19)--(-4,18)--(-5,17)--(-4,16)--(-3,17)--(-4,18)--(-3,19)--cycle (-4,20)--(-3,21);
\draw (-10,10)--(-11,11) (-11,11)--(-12,12)--(-11,13)--(-10,12)--cycle (-11,13)--(-10,14);
\draw (-1,23)--(0,24)--(1,25)--(2,24)--(1,23)--(0,22)--cycle (0,24)--(1,23);
\draw (11,11)--(12,12) (12,12)--(13,13)--(12,14)--(11,13)--cycle (12,14)--(11,15);
\draw[densely dotted] (1.75,1.75)--(1,1) (-0.75,2.75)--(0,2)--(0.75,2.75) (-1,1)--(-1.75,1.75);
\draw[densely dotted] (-2.25,2.25)--(-3,3)--(-2.25,3.75) (-2.25,4.25)--(-3,5)--(-2.25,5.75) (-2.25,6.25)--(-3,7)--(-2.25,7.75) (-3.25,8.75)--(-4,8)--(-4.75,8.75) (-5.75,7.75)--(-5,7)--(-5.75,6.25) (-5,5)--(-5.75,5.75);
\draw[densely dotted] (-3.25,9.25)--(-4,10)--(-4.75,9.25) (-5.75,10.25)--(-5,11)--(-5.75,11.75) (-5.75,12.25)--(-5,13)--(-5.75,13.75) (-4.75,14.75)--(-4,14)--(-3.25,14.75) (-2.25,12.25)--(-3,13)--(-2.25,13.75) (-2.25,10.25)--(-3,11)--(-2.25,11.75);
\draw[densely dotted] (-4.75,15.25)--(-4,16)--(-3.25,15.25) (-5.75,16.25)--(-5,17)--(-5.75,17.75) (-5.75,18.25)--(-5,19) (-2.25,18.25)--(-3,19)--(-2.25,19.75) (-2.25,16.25)--(-3,17)--(-2.25,17.75) (-2.25,21.75)--(-3,21)--(-2.25,20.25);
\draw[densely dotted] (-0.75,21.25)--(0,22)--(0.75,21.25) (1,23)--(1.75,22.25) (2,24)--(2.75,23.25) (-1.75,22.25)--(-1,23);
\draw[densely dotted] (-9.25,9.25)--(-10,10)--(-9.25,10.75) (-9.25,11.25)--(-10,12)--(-9.25,12.75) (-9.25,13.25)--(-10,14)--(-9.25,14.75);
\draw[densely dotted] (10.25,10.25)--(11,11)--(10.25,11.75) (10.25,12.25)--(11,13)--(10.25,13.75) (10.25,14.25)--(11,15)--(10.25,15.75);
\fill (0,0) circle (5pt) (-1,1) circle (5pt) (-3,3) circle (5pt) (-4,4) circle (5pt) (-5,5) circle (5pt) (-10,10) circle (5pt) (-11,11) circle (5pt) (-12,12) circle (5pt) (-10,12) circle (5pt);
\fill (-11,13) circle (5pt) (-10,14) circle (5pt) (-5,19) circle (5pt) (-4,20) circle (5pt) (-3,21) circle (5pt) (-1,23) circle (5pt) (0,24) circle (5pt) (1,25) circle (5pt);
\fill (-5,7) circle (5pt) (-5,11) circle (5pt) (-5,13) circle (5pt) (-5,17) circle (5pt);
\fill (-4,6) circle (5pt) (-4,8) circle (5pt) (-4,10) circle (5pt);
\draw[fill=white] (0,2) circle (5pt) (1,1) circle (5pt);
\draw[fill=white] (-3,5) circle (5pt) (-3,7) circle (5pt) (-3,11) circle (5pt) (-3,13) circle (5pt) (-3,17) circle (5pt) (-3,19) circle (5pt);
\draw[fill=white] (-4,12) circle (5pt) (-4,14) circle (5pt) (-4,16) circle (5pt) (-4,18) circle (5pt);
\draw[fill=white] (0,22) circle (5pt) (1,23) circle (5pt) (2,24) circle (5pt);
\draw[fill=white] (11,11) circle (5pt) (12,12) circle (5pt) (13,13) circle (5pt) (11,13) circle (5pt) (12,14) circle (5pt) (11,15) circle (5pt);
\draw [stealth-] (-3.7,10) -- (-1.3,10);
\draw [stealth-] (-3.7,12) -- (-1.3,12);
\node at (0,0) [below] {\tiny $(1,1)$};
\node at (-12,12) [left] {\tiny $(m,1)$};
\node at (13,13) [right] {\tiny $(1,n)$};
\node at (1,25) [above] {\tiny $(m,n)$};
\node at (-1.7,10) [right] {\tiny $(x,y)$};
\node at (-1.7,12) [right] {\tiny $(x^+,y^+)$};

\draw[decorate, decoration={snake, amplitude=0.5pt, segment length=2pt}] (21,7)--(24,10)--(22,12)--(23,13) (22,14)--(23,15)--(21,17)--(22,18);
\draw (28,0)--(20,8)--(24,12)--(21,15)--(23,17)--(22,18)--(29,25);
\fill (28,0) circle (5pt) (21,9) circle (5pt) (22,8) circle (5pt) (24,12) circle (5pt) (22,12) circle (5pt) (23,17) circle (5pt) (23,15) circle (5pt) (29,25) circle (5pt) (23,11) circle (5pt) (23,13) circle (5pt);
\node at (28,0) [below] {\tiny $(1,1)$};
\node at (29,25) [above] {\tiny $(m,n)$};
\node at (21,9) [left] {\tiny $\alpha_1$};
\node at (22,8) [right] {\tiny $\alpha_2$};
\node at (23,11) [right] {\tiny $\gamma_1$};
\node at (23,13) [right] {\tiny $\gamma_2$};
\node at (24,12) [right] {\tiny $\delta_1\!=\!(x^+, y^+)$};
\node at (22,12) [left] {\tiny $\delta_2$};
\node at (23,17) [right] {\tiny $\beta_1$};
\node at (23,15) [right] {\tiny $\beta_2$};

\draw (34,0)--(34,25);
\coordinate (A1) at (34,9) node at (A1) [right] {\tiny $\alpha_1$} ;
\coordinate (A2) at (34,12) node at (A2) [right] {\tiny $\delta_1$} ;
\coordinate (A3) at (34,17) node at (A3) [right] {\tiny $\beta_1$} ;
\coordinate (A4) at (34,11) node at (A4) [right] {\tiny $\gamma_1$};
\coordinate (A5) at (34,13) node at (A5) [right] {\tiny $\gamma_2$};
\fill (A1) circle (5pt) (A2) circle (5pt) (A3) circle (5pt) (A4) circle (5pt) (A5) circle (5pt);
\node at (34,0) [below] {\tiny $C$};

\draw (37,0)--(37,7) (37,11)--(37,13) (37,18)--(37,25);
\draw[decorate, decoration={snake, amplitude=0.5pt, segment length=2pt}] (37,7)--(37,11) (37,13)--(37,18);
\coordinate (B1) at (37,8) node at (B1) [right] {\tiny $\alpha_2$} ;
\coordinate (B2) at (37,12) node at (B2) [right] {\tiny $\delta_1$} ;
\coordinate (B3) at (37,15) node at (B3) [right] {\tiny $\beta_2$} ;
\coordinate (B4) at (37,11) node at (B4) [right] {\tiny $\gamma_1$};
\coordinate (B5) at (37,13) node at (B5) [right] {\tiny $\gamma_2$};
\fill (B1) circle (5pt) (B2) circle (5pt) (B3) circle (5pt)  (B4) circle (5pt) (B5) circle (5pt);
\node at (37,0) [below] {\tiny $D$};

\draw (43,0)--(43,7) (43,11)--(43,25);
\draw[decorate, decoration={snake, amplitude=0.5pt, segment length=2pt}] (43,7)--(43,11);
\coordinate (C1) at (43,8) node at (C1) [right] {\tiny $\alpha_2$} ;
\coordinate (C2) at (43,12) node at (C2) [right] {\tiny $\delta_1$} ;
\coordinate (C3) at (43,17) node at (C3) [right] {\tiny $\beta_1$} ;
\coordinate (C4) at (43,11) node at (C4) [right] {\tiny $\gamma_1$};
\coordinate (C5) at (43,13) node at (C5) [right] {\tiny $\gamma_2$};
\fill (C1) circle (5pt) (C2) circle (5pt) (C3) circle (5pt) (C4) circle (5pt) (C5) circle (5pt);
\node at (43,0) [below] {\tiny $E$};

\draw (40,0)--(40,7) (40,18)--(40,25);
\draw[decorate, decoration={snake, amplitude=0.5pt, segment length=2pt}] (40,7)--(40,18);
\coordinate (D1) at (40,8) node at (D1) [right] {\tiny $\alpha_2$} ;
\coordinate (D2) at (40,12) node at (D2) [right] {\tiny $\delta_2$} ;
\coordinate (D3) at (40,15) node at (D3) [right] {\tiny $\beta_2$} ;
\coordinate (D4) at (40,11) node at (D4) [right] {\tiny $\gamma_1$};
\coordinate (D5) at (40,13) node at (D5) [right] {\tiny $\gamma_2$};
\fill (D1) circle (5pt) (D2) circle (5pt) (D3) circle (5pt) (D4) circle (5pt) (D5) circle (5pt);
\node at (40,0.15) [below] {\tiny $D'$};
\end{tikzpicture}
\hspace*{2.5cm}(a) $S^{(x,y)}$ \hspace*{3.2cm} (b) Route map of $C, D, D', E$
\caption{$\bm{m} \times \bm{n}$ and maximal chains}
\label{fig6}
\end{figure}

\begin{proof}

Since $\#\mathrm{MaxChain}(\bm{m} \times \bm{n})=\binom{m+n-2}{m-1}$,
the theorem holds if $n$ or $m$ is equal to or less than $2$; because in this case, $\bm{m} \times \bm{n}$ does not have a crown structure.
Moreover, we have already seen in Example \ref{ex37} that the theorem holds for $m=n=3$.

Assume that both $m$ and $n$ are greater than or equal to $3$.
In what follows, see Figure \ref{fig6} if necessary.
Put $S^{(m, 1)}=\{(s, t) \in \bm{m} \times \bm{n} \mid t=1 \ \text{or}\ s=m \}$,
and for $(x, y) \in \{1, \dots, m-1\} \times \{2, \dots, n\}$, define
$$S^{(x, y)}=\{(s, t) \in \bm{m} \times \bm{n} \mid t-s < y-x, \ \text{or} \ t-s=y-x \ \text{and} \ t \leq y \} \cup S^{(m, 1)}$$
(the black points in Figure \ref{fig6}(a) are the elements of $S^{(x,y)}$).
Let $\mathcal{C}^{(x, y)}$ be the subset of $\mathrm{MaxChain}(\bm{m} \times \bm{n})$ consisting of elements of $S^{(x, y)}$, i.e.
$\mathcal{C}^{(x, y)}=\mathrm{MaxChain}(\bm{m} \times \bm{n}) \cap 2^{S^{(x, y)}}$.
Then $\mathcal{C}^{(x, y)}$ is an element of $\mathfrak{K}(\bm{m} \times \bm{n})$,
and especially, $\mathcal{C}^{(1, n)}=\mathrm{MaxChain}(\bm{m} \times \bm{n})$.

For $(x, y) \in ( \{1, \dots, m-1\} \times \{2, \dots, n\} ) \setminus \{(1, n)\}$,
we define
$$(x^+, y^+)=
   \begin{cases} (x - y+1,2) & \text{if} \ x = m-1 \ \text{or}\ y=n, \ \text{and} \ x > y, \\
                      (1, y-x+2) & \text{if} \ x = m-1 \ \text{or}\ y=n, \ \text{and} \ x \leq y, \\
                      (x+1, y+1) & \text{otherwise.}
   \end{cases}
$$
Then clearly $S^{(x, y)} \subsetneq S^{(x^+, y^+)}$ and $\mathcal{C}^{(x, y)} \subsetneq \mathcal{C}^{(x^+, y^+)}$ holds.

\begin{lemma}\label{lem322}
 $\mathcal{C}^{(x, y)} \subset \mathcal{C}^{(x^+, y^+)}$ is a covering in $\mathfrak{K}(\bm{m}\times\bm{n})$.
\end{lemma}

\begin{proof}
Fix $C \in \mathcal{C}^{(x^+, y^+)} \setminus \mathcal{C}^{(x, y)}$.
We will prove that $\overline{\mathcal{C}^{(x,y)} \cup \{C\}}=\mathcal{C}^{(x^+,y^+)}$.

 If $\#\left(\mathcal{C}^{(x^+, y^+)} \setminus \mathcal{C}^{(x, y)}\right)=1$, then obviously the lemma holds.
Thus suppose there exists $D \in \mathcal{C}^{(x^+, y^+)} \setminus \left(\mathcal{C}^{(x, y)} \cup \{C\}\right)$.
Put $\delta_1={(x^+, y^+)}, \delta_2=\{x^+ +1, y^+ -1\}$.
Note that $C, D$ both contain $\delta_1$ because $S^{(x^+, y^+)} \setminus S^{(x, y)}=\{\delta_1\}$.
Therefore, $\gamma_1:=(x^+, y^+ -1) \lessdot \delta_1 \lessdot \gamma_2:=(x^+ +1, y^+)$ is a covering sequence in both $C$ and $D$.
Put $D'=D \cup \{ \delta_2\}\setminus \{ \delta_1\}$; then $D' \in \mathcal{C}^{(x, y)}$.

Let $E=(-\infty, \gamma_1]_{D'} \cup (\gamma_1, \infty)_C$.
We first prove that $E$ is an element of $\overline{\mathcal{C}^{(x,y)} \cup \{C\}}$.
If $(-\infty,\gamma_1]_C=(-\infty, \gamma_1]_D$, then we have $E=C \in \overline{\mathcal{C}^{(x,y)} \cup \{C\}}$.
Otherwise, there exists elements 
$\alpha_1 \in (-\infty, \gamma_1]_C \setminus E$ and $\alpha_2 \in (-\infty, \gamma_1]_E \setminus C$.
Since
$$C \in \left(\mathcal{C}^{(x,y)} \cup \{C\}\right)_{\alpha_1 \delta_1}, D' \in \left(\mathcal{C}^{(x,y)} \cup \{C\}\right)_{\alpha_2 \delta_2}, \ \gamma_1 \in (\alpha_1, \delta_1)_C \cap (\alpha_2, \delta_2)_{D'},$$
$(\alpha_1, \delta_1, \alpha_2, \delta_2)$ is a star of $\mathcal{C}^{(x,y)} \cup \{C\}$.
Therefore, we obtain $E \in \overline{\mathcal{C}^{(x, y)} \cup \{C\}}$.

Secondly, we will prove $D \in \overline{\mathcal{C}^{(x,y)} \cup \{C\}}$ similarly.
Note that $E=(-\infty, \gamma_2]_{D} \cup (\gamma_2, \infty)_C$.
Thus if $(\gamma_2, \infty)_D=(\gamma_2, \infty)_E$, we have $D=E \in \overline{\mathcal{C}^{(x, y)} \cup \{C\}}$.
Otherwise, there exists elements
$\beta_1 \in (\gamma_2, \infty)_E \setminus D$ and $\beta_2 \in (\gamma_2, \infty)_D \setminus E.$
Since
$$E \in \left(\mathcal{C}^{(x,y)} \cup \{C\}\right)_{\delta_1 \beta_1}, D' \in \left(\mathcal{C}^{(x,y)} \cup \{C\}\right)_{\delta_2 \beta_2}, \ \gamma_2 \in (\delta_1, \beta_1)_E \cap (\delta_2, \beta_2)_{D'},$$
$(\delta_1, \beta_1, \delta_2, \beta_2)$ is a star of $\mathcal{C}^{(x,y)} \cup \{C\}$.
Thus, we have
$$D=(-\infty, \gamma_2]_{E} \cup (\gamma_2, \infty )_{D'} \in \overline{\mathcal{C}^{(x,y)} \cup \{C\}}.$$ 

Therefore, every element $D$ of $\mathcal{C}^{(x^+, y^+)}$ is also an element of $\overline{\mathcal{C}^{(x,y)} \cup \{C\}}$, i.e.
$\mathcal{C}^{(x^+, y^+)} \subset \overline{\mathcal{C}^{(x,y)} \cup \{C\}} \subset \overline{\mathcal{C}^{(x^+, y^+)}}$.
Since $\overline{\mathcal{C}^{(x^+, y^+)}}=\mathcal{C}^{(x^+, y^+)}$, we have $\mathcal{C}^{(x^+, y^+)} = \overline{\mathcal{C}^{(x,y)} \cup \{C\}}$, 
and the lemma follows from Corollary \ref{cor319}.
\end{proof}

\it{Proof of Theorem \ref{thm321} continued.}\
Note that $\#\mathcal{C}^{(m,1)}=1$ and $\#\mathcal{C}^{(m-1,2)}=2$.
Thus, Lemma \ref{lem322} yields that the inclusion sequence
\begin{multline*}
\emptyset \subset \mathcal{C}^{(m, 1)} \subset \mathcal{C}^{(m-1, 2)} 
    \subset \mathcal{C}^{(m-2, 2)} \subset \mathcal{C}^{(m-1, 3)}  \subset \mathcal{C}^{(m-3, 2)} \subset \dots\\
     \subset \mathcal{C}^{(x, y)} \subset \mathcal{C}^{(x^+, y^+)} \subset \dots 
    \subset \mathcal{C}^{(2, n)} \subset \mathcal{C}^{(1, n)}=\mathrm{MaxChain}(\bm{m} \times \bm{n})
\end{multline*}
is a maximal chain of $\mathfrak{K}(\bm{m} \times \bm{n})$.
Thus, the theorem follows from Corollary \ref{cor316}.
\end{proof}

\section*{Acknowledgement}
The author would like to thank Masanori Kobayashi for numerous helpful discussions.

\bibliographystyle{plain}
\nocite{*} 
\bibliography{crown20210824}

\end{document}